\documentclass[11pt, english]{article}
\usepackage[T1]{fontenc}
\usepackage[utf8]{inputenc}

\usepackage{graphicx}
\usepackage{textcomp}
\usepackage{babel}
\usepackage{tikz}
\usetikzlibrary{arrows.meta,calc,decorations.pathreplacing}
\usepackage{amsmath,amssymb,amsthm,mathrsfs}
\usepackage{xcolor}

\DeclareMathAlphabet{\mathpzc}{OT1}{pzc}{m}{it}
\usepackage{hyperref}
\usepackage[nameinlink]{cleveref}
\hypersetup{
	colorlinks=true,
	linkcolor=blue!60!black,
	citecolor=blue!60!black,
	urlcolor=blue!60!black
}

\usepackage{aliascnt}

\theoremstyle{plain}
\newtheorem{theorem}{Theorem}[section]

\newaliascnt{proposition}{theorem}
\newtheorem{proposition}[proposition]{Proposition}
\aliascntresetthe{proposition}

\newaliascnt{lemma}{theorem}
\newtheorem{lemma}[lemma]{Lemma}
\aliascntresetthe{lemma}

\newaliascnt{corollary}{theorem}
\newtheorem{corollary}[corollary]{Corollary}
\aliascntresetthe{corollary}

\newaliascnt{assumption}{theorem}

\aliascntresetthe{assumption}

\theoremstyle{definition}
\newaliascnt{definition}{theorem}
\newtheorem{definition}[definition]{Definition}
\aliascntresetthe{definition}

\newaliascnt{remark}{theorem}
\newtheorem{remark}[remark]{Remark}
\aliascntresetthe{remark}

\newaliascnt{example}{theorem}

\aliascntresetthe{example}

\DeclareMathOperator{\supp}{spt}

\numberwithin{equation}{section}
\usepackage{authblk}

\title{\bf A Fully Discrete Variational Approximation of Mather Measures and Sets}

\author[1]{Fabio Camilli}
\author[2]{Cristian Mendico}
\affil[1]{Dip. di Ingegneria e Geologia, Univ. "G. D'Annunzio" Chieti-Pescara, viale Pindaro 42, 65127 Pescara (Italy) \\ \href{mailto:fabio.camilli@unich.it}{fabio.camilli@unich.it}}
\affil[2]{Institut de Math\'ematique de Bourgogne - UMR 5584 CNRS, Universit\'e Bourgogne Europe\\ \href{mailto:cristian.mendico@u-bourgogne.fr}{cristian.mendico@u-bourgogne.fr}}

\begin{document}
	\maketitle

	\begin{abstract}
		We introduce a fully--discrete variational approximation of Mather measures and sets for Tonelli Lagrangians on the flat torus, together with a numerical procedure for approximating the entire Mather set. The scheme is based on a fully--discrete Lax--Oleinik operator with integer winding labels. We prove an $O(\tau+h/\tau)$ error estimate for the critical value, convergence of critical solutions, and a finite-dimensional characterization of fully--discrete Mather measures. Accumulation points of the reconstructed minimizing measures are continuous Mather measures, while the supports satisfy complementary upper and lower convergence results involving the Ma\~n\'e and Mather sets. To avoid the selection of only some minimizing components by exact discrete minimizers, we introduce a mass-threshold approximation based on almost-minimizing holonomic measures. This yields a finite-dimensional constrained optimization procedure designed to recover the whole Mather set.

		\vspace{0.25cm}
		\noindent\textbf{Key words:} Aubry--Mather and weak KAM theory; Ma\~n\'e set; fully--discrete Lax--Oleinik schemes; Mather measures; holonomic measures; Hamilton--Jacobi equations; Kuratowski convergence.
		\\
		\noindent\textbf{2020 AMS:} 37J50; 49L25; 35F21; 65M12.
	\end{abstract}

\section{Introduction}

Aubry--Mather theory describes the long-time behavior of Tonelli
Lagrangian systems through action-minimizing measures and their
supports. These objects are closely related to critical solutions of
the stationary Hamilton--Jacobi equation
\[
H(x,Du)=\alpha(H)
\qquad\text{in }\mathbb T^d,
\]
and encode the minimizing dynamics of the associated Euler--Lagrange
flow. We refer to Mather's foundational work \cite{Mather1991} and to
\cite{Fathi2008,Sorrentino2015} for comprehensive accounts of
Aubry--Mather and weak KAM theory.

Discrete variational models and their convergence to the continuous
theory have been studied in
\cite{Garibaldi,Gomes,Su}. Related variational, numerical, and discrete
weak KAM approaches can be found in
\cite{Camilli,Gomes_Ob,Hadikhanloo,Iturriaga-Mendico,
	Iturriaga-Wang,Rorro,Soga,tran,Zavidovique}. Numerical methods for invariant
structures are also well developed in the classical KAM setting,
including parameterization methods for invariant Lagrangian tori
\cite{HaroDeLaLlave,HuguetDeLaLlaveSire,HaroEtAl,
	FiguerasHaroLuque} and Greene-type or converse KAM methods for their
breakdown and nonexistence
\cite{MacKayPercival,MacKayMeissStark,MacKayGreene}.

In our previous work \cite{CamilliMendico}, we considered a
semi--discrete approximation in which time is discretized while the
state variable remains continuous. Here we develop a fully--discrete
counterpart and, more importantly, a finite-dimensional procedure for
approximating the whole Mather set. To the best of our knowledge, such
a computational procedure for the entire Mather set has not previously
been established in this generality.

Let $\mathcal G_h$ be a uniform grid of the flat torus and let
$\tau>0$ be the time step. Each transition between grid points
$x_i,x_j\in\mathcal G_h$ is equipped with a winding label
$\ell\in\mathbb Z^d$ and velocity
\[
v_{ij}^{\ell}:=\frac{x_j+\ell-x_i}{\tau}.
\]
The fully--discrete Lax--Oleinik operator is
\[
(T_{h,\tau}u_h)(x_j)
:=
\min_{\substack{x_i\in\mathcal G_h\\ \ell\in\mathbb Z^d}}
\left\{
u_h(x_i)
+
\tau L\left(x_i,\frac{x_j+\ell-x_i}{\tau}\right)
\right\}.
\]
The winding label distinguishes transitions with the same endpoints
on the torus but different lifts to the universal covering space. This
information is essential for reconstructing measures and minimizing
sets in phase space.

A closely related fully--discrete Lax--Oleinik integrator was studied
by Bouillard, Faou, and Zavidovique
\cite{BouillardFaouZavidovique}. Their quotient cost minimizes over
the lifts of a torus transition, whereas our formulation retains every
lift as a separate winding-labeled transition. Although the two
representations are related at the operator level, the latter preserves
the phase-space information needed to define fully--discrete holonomic
measures and study the convergence of their supports.

We first analyze the critical equation
\[
T_{h,\tau}u_{h,\tau}
=
u_{h,\tau}+\tau\bar L(h,\tau).
\]
We prove existence and uniqueness of the critical value, uniform
estimates for critical solutions and minimizing velocities, and the
error bound
\[
\bigl|\bar L(h,\tau)+\alpha(H)\bigr|
\le
C\left(\tau+\frac h\tau\right).
\]
The proof adapts the long-time comparison argument of
\cite{BouillardFaouZavidovique} to the present left-endpoint cost and
winding representation. Suitable interpolations of normalized critical
solutions then converge, up to subsequences, to viscosity solutions of
the continuous critical Hamilton--Jacobi equation.

We next formulate the fully--discrete problem in terms of holonomic
measures. Let $\mathcal H_{h,\tau}$ be the set of probability
distributions on the labeled transitions $(x_i,x_j,\ell)$ satisfying
mass conservation at every grid point. We prove that
\[
\bar L(h,\tau)
=
\min_{\mu\in\mathcal H_{h,\tau}}
\sum_{i,j,\ell}
L\left(x_i,\frac{x_j+\ell-x_i}{\tau}\right)
\mu_{ij}^{\ell}.
\]
The minimizers are called fully--discrete Mather measures, and their
supports define the fully--discrete Mather set. In graph terms, grid
points are vertices, labeled transitions are weighted directed edges,
and holonomy is the mass-balance condition. The critical equation is
therefore a min--plus eigenvalue problem related to minimum-mean cycles
and normalized circulations \cite{karp,Zavidovique}. This
finite-dimensional graph structure underlies our numerical procedure.

After reconstructing the discrete measures on
$\mathbb T^d\times\mathbb R^d$, we prove that every narrow accumulation
point of fully--discrete Mather measures is a continuous Mather
measure. For their supports, we establish the upper Kuratowski
inclusion
\[
\limsup_{n\to\infty}
\widetilde{\mathcal M}_{h_n,\tau_n}
\subset
\widetilde{\mathcal N}_L,
\]
where $\widetilde{\mathcal N}_L$ is the continuous Mañé set. Under
uniqueness of the continuous Mather measure, we also obtain
\[
\widetilde{\mathcal M}_L
\subset
\liminf_{n\to\infty}
\widetilde{\mathcal M}_{h_n,\tau_n}.
\]
Consequently, if
$\widetilde{\mathcal N}_L=\widetilde{\mathcal M}_L$, then
$\widetilde{\mathcal M}_{h_n,\tau_n}$ converges in the Kuratowski sense
to $\widetilde{\mathcal M}_L$.

Exact discrete minimizers may nevertheless select only some continuous
minimizing components, since small grid-dependent errors can separate
components having the same continuous action. To overcome this
selection effect, we introduce a mass-threshold approximation based on
almost-minimizing holonomic measures. An action tolerance retains
components with nearly minimal discrete action, while a local mass
threshold removes regions carrying negligible mass.

Our main computational result is that this approximation detects
\emph{every} minimizing component of the continuous system, at any
prescribed resolution in phase space. To the best of our knowledge, this
is the first scheme that provably computes the whole Mather set: previous
discrete methods approximate the critical value, the critical solutions,
or a single minimizing object. The procedure is completely explicit. All
the parameters it needs are given in closed form: the action tolerance
comes from the critical-value estimate and is of order $O(\tau)$ in the
balanced regime $h\asymp\tau^{2}$. What one has to solve is a finite
family of constrained minimization problems on a finite graph, with linear
constraints. We illustrate the mechanism on an example with two minimizing
components, where exact discrete minimization finds only one of them,
while our procedure finds both. Being able to see the whole Mather set, and not just a component,
opens the way to the numerical study of its finer structure. We plan to
return to this in future work.

The paper is organized as follows. Section~2 recalls the continuous and
semi--discrete settings. Sections~3--4 introduce the fully--discrete
scheme and its variational formulation. Section~5 studies the
convergence of critical solutions, measures, and Mather sets.
Section~6 introduces the mass-threshold approximation, and Section~7
describes its numerical implementation and an illustrative example.

	\section{Preliminaries on the continuous  Mather set}
	Let $\mathbb{T}^{d}:=\mathbb{R}^{d}/\mathbb{Z}^{d}$ be the flat torus and let $H:\mathbb{T}^{d}\times\mathbb{R}^{d}\to\mathbb{R}$ be a Tonelli Hamiltonian, i.e. $H\in C^{2}(\mathbb{T}^{d}\times\mathbb{R}^{d})$,   strictly convex in the momentum variable, namely
	\[
	D^{2}_{pp}H(x,p)>0 \qquad \text{for every } (x,p)\in \mathbb{T}^{d}\times\mathbb{R}^{d},
	\]
	and   superlinear in $p$. We denote by
	$L(x,v):=\sup_{p\in\mathbb{R}^{d}}\{p\cdot v-H(x,p)\}$
	the associated Tonelli Lagrangian. Thus $L\in C^{2}(\mathbb{T}^{d}\times\mathbb{R}^{d})$, it is strictly convex in $v$ and superlinear, that is, for every $K>0$ there exists $C(K)\in\mathbb{R}$ such that
	$L(x,v)\geq K|v|-C(K)$ with $(x,v)\in\mathbb{T}^{d}\times\mathbb{R}^{d}.$

	To introduce the definition of Mather set, we recall first the definition of closed probability measures. A probability measure	$\mu\in \mathcal{P}(\mathbb{T}^{d}\times\mathbb{R}^{d})$
	is said to be closed if
	\[
	\int_{\mathbb{T}^{d}\times\mathbb{R}^{d}} v\cdot D\phi(x)\,d\mu(x,v)=0
	\qquad \forall \phi\in C^{1}(\mathbb{T}^{d}).
	\]
	A closed probability measure $\mu$ is called a Mather measure if it minimizes the average action
	$\int_{\mathbb{T}^{d}\times\mathbb{R}^{d}} L(x,v)\,d\mu(x,v)$
	among all closed probability measures. Equivalently, if $\alpha(H)$ denotes the critical value appearing in the stationary Hamilton--Jacobi equation
	\begin{equation}\label{HJ}
		H(x,Du(x))=\alpha(H),\qquad x\in\mathbb{T}^{d},
	\end{equation}
	then Mather measures are closed measures attaining the critical action level
	\[
	\min_{\mu\ \mathrm{closed}}
	\int_{\mathbb{T}^{d}\times\mathbb{R}^{d}} L(x,v)\,d\mu(x,v)
	=
	-\alpha(H).
	\]

	The continuous Mather set is defined as the closure of the union of the supports of all Mather measures:
	\[
	\widetilde{\mathcal{M}}_{L}
	:=
	\overline{
		\bigcup
		\left\{
		\supp(\mu):
		\mu \ \text{is a Mather measure}
		\right\}
	}
	\subset \mathbb{T}^{d}\times\mathbb{R}^{d}.
	\]
	It is a nonempty compact set, invariant under the Euler--Lagrange flow
	$\Phi^{t}_{L}$; recall that every minimizing closed measure is in fact
	invariant under $\Phi^{t}_{L}$, see \cite{Fathi2008,Sorrentino2015}.

	We shall also use the continuous  Mañé  and Aubry set. Let $\mathcal S^{-}$ denote the
	set of viscosity solutions of \eqref{HJ}. A curve
	$\gamma:\mathbb R\to\mathbb T^{d}$ is said to be
	$(u,L,\alpha(H))$--calibrated if
	\begin{equation}\label{eq:calibrated-curve}
		u(\gamma(b))-u(\gamma(a))
		=
		\int_{a}^{b}L(\gamma(s),\dot\gamma(s))\,ds+\alpha(H)(b-a)
		\qquad\text{for all }a\le b,
	\end{equation}
	and we set
	\[
	\widetilde{\mathcal I}(u)
	:=
	\bigl\{(\gamma(t),\dot\gamma(t)):
	\gamma\ (u,L,\alpha(H))\text{--calibrated},\ t\in\mathbb R\bigr\}.
	\]
	The Aubry and the Ma\~n\'e sets are then defined by
	\begin{equation}\label{eq:aubry-mane-sets}
		\widetilde{\mathcal A}_{L}
		:=
		\bigcap_{u\in\mathcal S^{-}}\widetilde{\mathcal I}(u),
		\qquad
		\widetilde{\mathcal N}_{L}
		:=
		\bigcup_{u\in\mathcal S^{-}}\widetilde{\mathcal I}(u),
	\end{equation}
	and they satisfy
	$\widetilde{\mathcal M}_{L}\subset\widetilde{\mathcal A}_{L}
	\subset\widetilde{\mathcal N}_{L}$; see \cite{Fathi2008,Sorrentino2015}.
	In particular the trajectories of $\widetilde{\mathcal M}_{L}$ are globally
	minimizing.

	\section{A Fully Discrete Lax--Oleinik Scheme}
	In this section, we introduce a fully--discrete Lax--Oleinik scheme, in the spirit of the space discretization briefly outlined in \cite[Section 6]{Iturriaga-Wang}. It can also be viewed as the restriction of the semi--discrete scheme studied in \cite{CamilliMendico} to the spatial grid $\mathcal G_h$ and to the corresponding discrete velocity set generated by the lifted grid displacements. We study the basic properties of the resulting operator and its consistency with the semi--discrete problem.

	Let $\tau>0$ be a time step and let $h>0$ be a space discretization parameter.
	We consider the uniform grid of the flat torus
	\[
	\mathcal G_h
	:=
	\left\{
	x_i=ih \pmod{\mathbb Z^d}
	:\;
	i=(i_1,\dots,i_d)\in\{0,\dots,N-1\}^d
	\right\},
	\qquad Nh=1.
	\]
	We identify each grid point with its representative in $[0,1)^d$. For $x_i,x_j\in\mathcal G_h$ and $\ell\in\mathbb Z^d$, we define the discrete
	velocity
	\[
	v_{ij}^{\ell} := \frac{x_j+\ell-x_i}{\tau}.
	\]
	The integer vector $\ell\in\mathbb Z^d$ selects the lifted displacement
	$x_j+\ell-x_i$ and hence distinguishes the periodic realizations of the same
	transition on the torus; see Figure~\ref{fig:integer-lift-transition}.
	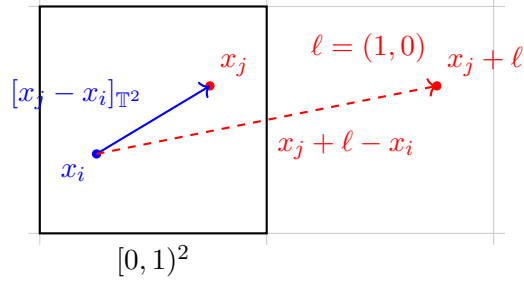
\begin{figure}[ht]
		\centering
		\begin{tikzpicture}[scale=3]

			\draw[step=1,gray!40,thin] (-0.05,-0.05) grid (2.05,1.05);

			\draw[thick] (0,0) rectangle (1,1);
			\node at (0.5,-0.12) {$[0,1)^2$};

			\filldraw[blue] (0.25,0.35) circle (0.018) node[below left] {$x_i$};
			\filldraw[red] (0.75,0.65) circle (0.018) node[above right] {$x_j$};

			\filldraw[red] (1.75,0.65) circle (0.018) node[above right] {$x_j+\ell$};

			\draw[->,blue,thick] (0.25,0.35) -- (0.75,0.65)
			node[midway,above left] {$[x_j-x_i]_{\mathbb T^2}$};

			\draw[->,red,thick,dashed] (0.25,0.35) -- (1.75,0.65)
			node[midway,below right] {$x_j+\ell-x_i$};

			\node[red] at (1.45,0.82) {$\ell=(1,0)$};


		\end{tikzpicture}
		\caption{Different lifts of the same transition on the torus. The point
			$x_j+\ell$ is a copy of $x_j$ in the universal covering space, and the vector
			$\ell$ selects the corresponding periodic representative of the displacement.}
		\label{fig:integer-lift-transition}
	\end{figure}
	For $x_i,x_j\in\mathcal G_h$ and $\ell\in\mathbb Z^d$, we define the lifted
	fully--discrete one-step action by
	\begin{equation}\label{eq:fd-lifted-action-winding}
		\widehat A_{\tau}^h(x_i,x_j,\ell)
		:=
		\tau L\left(x_i,v_{ij}^{\ell}\right)
		=
		\tau L\left(x_i,\frac{x_j+\ell-x_i}{\tau}\right).
	\end{equation}
	The corresponding one-step action between grid points on the torus is
	\begin{equation}\label{eq:fd-action}
		A_{\tau}^h(x_i,x_j)
		:=
		\min_{\ell\in\mathbb Z^d}\widehat A_{\tau}^h(x_i,x_j,\ell).
	\end{equation}
	We define the fully--discrete Lax--Oleinik operator by the equivalent formulas
	\begin{equation}\label{eq:fully-discrete-operator}
		\begin{aligned}
			(T_{h,\tau}u_h)(x_j)
			:&=
			\min_{x_i\in\mathcal G_h,\ \ell\in\mathbb Z^d}
			\left\{u_h(x_i)+\widehat A_{\tau}^h(x_i,x_j,\ell)\right\}\\
			&=
			\min_{x_i\in\mathcal G_h}
			\left\{u_h(x_i)+A_{\tau}^h(x_i,x_j)\right\},
			\qquad x_j\in\mathcal G_h.
		\end{aligned}
	\end{equation}
	The minimum is attained, since $\mathcal G_h$ is finite and $L$ is superlinear
	in the velocity variable.

	We say that a pair
	$(u_{h,\tau},\bar L(h,\tau)) \in \mathbb R^{\mathcal G_h}\times\mathbb R$
	solves the fully--discrete ergodic Lax--Oleinik equation if
	\begin{equation}\label{eq:fully-discrete-ergodic-equation}
		u_{h,\tau}(x_j)+\tau\bar L(h,\tau)
		=
		(T_{h,\tau}u_{h,\tau})(x_j),
		\qquad x_j\in\mathcal G_h .
	\end{equation}

	\begin{proposition}
		\label{prop:existence-fully-discrete}
		For every $h>0$ and $\tau>0$, there exist
		$\bar L(h,\tau)\in\mathbb R$ and
		$u_{h,\tau}:\mathcal G_h\to\mathbb R$ solving
		\eqref{eq:fully-discrete-ergodic-equation}. Moreover, the constant
		$\bar L(h,\tau)$ is uniquely determined and satisfies
		\[
		\inf_{\mathbb T^d\times\mathbb R^d}L
		\le
		\bar L(h,\tau)
		\le
		\max_{x\in\mathbb T^d}L(x,0).
		\]
	\end{proposition}

	\begin{proof}
		The operator $T_{h,\tau}$ is monotone and additively homogeneous:
		\[
		u_h\le w_h
		\quad\Longrightarrow\quad
		T_{h,\tau}u_h\le T_{h,\tau}w_h,
		\quad \textrm{and} \quad
		T_{h,\tau}(u_h+c)=T_{h,\tau}u_h+c.
		\]
		Moreover, the minimum in \eqref{eq:fully-discrete-operator} is attained by the
		superlinearity of $L$.

		We prove existence by a discounted approximation. For $\delta>0$, consider
		\[
		\delta u_h^\delta(x_j)+u_h^\delta(x_j)
		=
		(T_{h,\tau}u_h^\delta)(x_j),
		\qquad x_j\in\mathcal G_h.
		\]
		The map
		\[
		u_h\mapsto \frac{1}{1+\delta}T_{h,\tau}u_h
		\]
		is a contraction on $\mathbb R^{\mathcal G_h}$, hence the discounted equation
		has a unique solution. Fix $x_0\in\mathcal G_h$ and set
		$v_h^\delta(x_i):=u_h^\delta(x_i)-u_h^\delta(x_0).$
		Given $h,\tau$, the normalized family $(v_h^\delta)_\delta$ has uniformly
		bounded oscillation. Since $\mathcal G_h$ is finite, any two grid points can be connected by a finite-cost path. Iterating the discounted dynamic programming inequality along this path gives
		$\operatorname{osc}_{\mathcal G_h} v_h^\delta \le C_{h,\tau},$ where $C_{h,\tau}$ is independent of $\delta$. Therefore, up to a sequence
		$\delta_n\to0$,
		$v_h^{\delta_n}\to u_{h,\tau}$ in  $\mathbb R^{\mathcal G_h}.$
		Moreover, evaluating the discounted equation at a maximum and at a
		minimum point of $u_h^\delta$ and arguing as in the last part of the
		proof gives
		\[
		\tau\inf_{\mathbb T^d\times\mathbb R^d}L
		\le
		\delta u_h^\delta(x_i)
		\le
		\tau\max_{x\in\mathbb T^d}L(x,0),
		\qquad x_i\in\mathcal G_h,
		\]
		so that, up to extracting a further subsequence,
		$\delta_n u_h^{\delta_n}(x_0)\to \lambda_{h,\tau}$. Substituting
		$u_h^\delta=v_h^\delta+u_h^\delta(x_0)$ into the discounted equation
		gives
		\[
		\delta v_h^{\delta}+v_h^{\delta}+\delta u_h^{\delta}(x_0)
		=
		T_{h,\tau}v_h^{\delta},
		\]
		and passing to the limit
		$T_{h,\tau}u_{h,\tau}=u_{h,\tau}+\lambda_{h,\tau}.$
		Setting $\lambda_{h,\tau}:=\tau\bar L(h,\tau)$ yields
		\eqref{eq:fully-discrete-ergodic-equation}.

		Uniqueness of the constant follows from the standard maximum principle. If
		$T_{h,\tau}u=u+\tau c$, $T_{h,\tau}w=w+\tau c',$
		and
		$M:=\displaystyle{\max_{\mathcal G_h}}(u-w)$,
		then $u\le w+M$. By monotonicity and additivity,
		$T_{h,\tau}u\le T_{h,\tau}w+M.$
		Evaluating at a point where $u-w$ attains its maximum gives $c\le c'$. Reversing
		the roles of $u$ and $w$ gives $c'=c$.

		It remains to produce a bound for $\bar L(h,\tau)$. Let $x_M$ and $x_m$ be maximum and
		minimum points of $u_{h,\tau}$. Since the transition $x_M\to x_M$ with
		$\ell=0$ is admissible,
		\[
		u_{h,\tau}(x_M)+\tau\bar L(h,\tau)
		=
		T_{h,\tau}u_{h,\tau}(x_M)
		\le
		u_{h,\tau}(x_M)+\tau L(x_M,0),
		\]
		so
		$\bar L(h,\tau)\le \max_{x\in\mathbb T^d}L(x,0)$.
		On the other hand, for every $x_i,x_j,\ell$,
		\[
		\widehat A_{\tau}^h(x_i,x_j,\ell)
		\ge
		\tau\inf_{\mathbb T^d\times\mathbb R^d}L.
		\]
		Using \eqref{eq:fully-discrete-ergodic-equation} at the minimum point $x_m$ gives
		\[
		u_{h,\tau}(x_m)+\tau\bar L(h,\tau)
		\ge
		u_{h,\tau}(x_m)+\tau\inf_{\mathbb T^d\times\mathbb R^d}L.
		\]
		Therefore
		$\inf_{\mathbb T^d\times\mathbb R^d}L
		\le
		\bar L(h,\tau)
		\le
		\max_{x\in\mathbb T^d}L(x,0)$.
	\end{proof}

	We next establish uniform bounds for critical solutions and minimizing velocities, following the strategy of \cite[Prop. 4]{Iturriaga-Wang}.

	\begin{proposition}
		\label{prop:fd-lip-bound}
		Let $\tau_0>0$ and assume that
		$0<\tau\le\tau_0$ and $h\le\eta_0\tau$,
		for some fixed positive constant $\eta_0$. Then there exist constants
		$C,D>0$, independent of $h$ and $\tau$, such that every solution
		$u_{h,\tau}$ of \eqref{eq:fully-discrete-ergodic-equation} satisfies
		\[
		|u_{h,\tau}(x_i)-u_{h,\tau}(x_j)|
		\le
		C\,d_{\mathbb T^d}(x_i,x_j),
		\qquad x_i,x_j\in\mathcal G_h .
		\]
		Moreover, if $(x_i,\ell)$ realizes the minimum in
		\[
		u_{h,\tau}(x_j)+\tau\bar L(h,\tau)
		=
		u_{h,\tau}(x_i)+\widehat A_{\tau}^h(x_i,x_j,\ell),
		\]
		then
		\begin{equation}\label{eq:bounded-jump}
			|v_{ij}^{\ell}|
			=
			\left|
			\frac{x_j+\ell-x_i}{\tau}
			\right|
			\le D.
		\end{equation}
	\end{proposition}
	\begin{proof}
		Set $\overline A_{\tau}^h:=\tau\bar L(h,\tau).$ 	By Proposition \ref{prop:existence-fully-discrete}, there exists \(B>0\),
		independent of \(h\) and \(\tau\), such that
		$|\bar L(h,\tau)|\le B.$
		We divide the proof into four steps.

		\emph{Step 1: action bounds.} With $v_{ij}^\ell:=(x_j+\ell-x_i)/\tau$, for every $D_0>0$ there is $C(D_0)>0$, independent of $h,\tau$, such that
		\begin{equation}\label{eq:bounded-small-jumps}
			\widehat A_{\tau}^h(x_i,x_j,\ell)-\overline A_{\tau}^h \le C(D_0)\tau
		\end{equation}
		whenever $|v_{ij}^\ell|\le D_0$. Indeed,
		\[ \widehat A_{\tau}^h(x_i,x_j,\ell)-\overline A_{\tau}^h = \tau\bigl[L(x_i,v_{ij}^\ell)-\bar L(h,\tau)\bigr], \] and the claim follows from the boundedness of $\bar L(h,\tau)$ and the continuity of $L$ on the compact set $\mathbb T^d\times\overline B_{D_0}$. We next derive a coercive lower bound. By the superlinearity of $L$, for every $M>0$ there exists $C_M>0$ such that
		$L(x,v)\ge M|v|-C_M$ for any $(x,v)\in\mathbb T^d\times\mathbb R^d.$
		Since $|\bar L(h,\tau)|\le B$, it follows that
		\begin{equation}\label{eq:coercive-action} \widehat A_{\tau}^h(x_i,x_j,\ell)-\overline A_{\tau}^h \ge M|x_j+\ell-x_i|-\tau(C_M+B).
		\end{equation}
		Consequently, \[ \frac{ \widehat A_{\tau}^h(x_i,x_j,\ell)-\overline A_{\tau}^h }{ |x_j+\ell-x_i| } \longrightarrow+\infty \] uniformly as $|v_{ij}^\ell| = \frac{|x_j+\ell-x_i|}{\tau} \longrightarrow+\infty.$

		\emph{Step 2: coarse increment estimate.} Equation \eqref{eq:fully-discrete-ergodic-equation} implies the subsolution inequality
		\begin{equation}\label{eq:fd-subsolution} u_{h,\tau}(x_j)-u_{h,\tau}(x_i) \le \widehat A_{\tau}^h(x_i,x_j,\ell)-\overline A_{\tau}^h
		\end{equation} for every $x_i,x_j\in\mathcal G_h$ and $\ell\in\mathbb Z^d$.\par
		 Suppose first that
		$d_{\mathbb T^d}(x_i,x_j)\le\tau.$
		Choose $q\in\mathbb Z^d$ such that
		$|x_j+q-x_i|=d_{\mathbb T^d}(x_i,x_j).$
		The corresponding discrete velocity has norm at most one. Therefore, applying \eqref{eq:fd-subsolution} with $\ell=q$ and using \eqref{eq:bounded-small-jumps} with $D_0=1$, we obtain
		\begin{equation}\label{eq:short-distance-estimate}
			u_{h,\tau}(x_j)-u_{h,\tau}(x_i) \le C_0\tau.
		\end{equation}
		Suppose now that $d_{\mathbb T^d}(x_i,x_j)>\tau.$ Since $h\le\eta_0\tau$, there exists a grid chain $ x_i=z_0,z_1,\ldots,z_N=x_j$  such that
		$d_{\mathbb T^d}(z_m,z_{m+1})\le\tau$, with $m=0,\ldots,N-1,$
		and  $N\tau\le C_1d_{\mathbb T^d}(x_i,x_j)$,  where $C_1$ depends only on the dimension and on $\eta_0$. Applying \eqref{eq:short-distance-estimate} along the chain and summing, we find \[
		u_{h,\tau}(x_j)-u_{h,\tau}(x_i) = \sum_{m=0}^{N-1} \bigl(u_{h,\tau}(z_{m+1})-u_{h,\tau}(z_m)\bigr)\le C_0N\tau\le C_2d_{\mathbb T^d}(x_i,x_j).
		\]
		Hence, if   $d_{\mathbb T^d}(x_i,x_j)>\tau$, then
		\begin{equation}\label{eq:large-scale-lip}
			u_{h,\tau}(x_j)-u_{h,\tau}(x_i) \le C_2d_{\mathbb T^d}(x_i,x_j).
		\end{equation}

		\emph{Step 3: minimizing velocities.}
		Let \((x_i,\ell)\) be a minimizer in \eqref{eq:fully-discrete-ergodic-equation} at \(x_j\). Set
		$v:=v_{ij}^{\ell} = \frac{x_j+\ell-x_i}{\tau}.$
		The calibration identity reads
		\begin{equation}\label{eq:calibrated-edge}
			\widehat A_{\tau}^h(x_i,x_j,\ell)-\overline A_{\tau}^h
			=
			u_{h,\tau}(x_j)-u_{h,\tau}(x_i).
		\end{equation}
		We distinguish two cases.	If
		$d_{\mathbb T^d}(x_i,x_j)\le\tau$,
		then \eqref{eq:short-distance-estimate} and
		\eqref{eq:calibrated-edge} yield
		$\tau L(x_i,v)-\tau\bar L(h,\tau) \le C_0\tau.$
		Therefore
		\begin{equation}\label{eq:min-velocity-short}
			L(x_i,v)\le C_0+B.
		\end{equation}
		By the superlinearity of \(L\), this implies a uniform bound on \(|v|\).

		If instead
		$d_{\mathbb T^d}(x_i,x_j)>\tau$,
		then \eqref{eq:large-scale-lip} and \eqref{eq:calibrated-edge} imply
		$\tau L(x_i,v)-\tau\bar L(h,\tau) \le C_2d_{\mathbb T^d}(x_i,x_j).$
		Since 	$d_{\mathbb T^d}(x_i,x_j)	\le	|x_j+\ell-x_i|
		=\tau|v|$,
		we obtain
		\begin{equation}\label{eq:min-velocity-long}
			L(x_i,v)\le C_2|v|+B.
		\end{equation}
		Choose \(M>C_2\) in the superlinearity estimate
		$L(x,v)\ge M|v|-C_M.$
		Combining it with \eqref{eq:min-velocity-long} gives
		$(M-C_2)|v|\le B+C_M.$
		Hence \(|v|\) is uniformly bounded also in this case.
		It follows that there exists \(D>0\), independent of \(h\) and \(\tau\),
		such that \eqref{eq:bounded-jump} holds. 	\par
		\emph{Step 4: local Lipschitz estimate.}
		Let \(x_j,x_m\in\mathcal G_h\) satisfy $d_{\mathbb T^d}(x_j,x_m)\le\tau.$
		Choose a minimizing pair \((x_i,\ell)\) in \eqref{eq:fully-discrete-ergodic-equation} at \(x_j\).
		By \eqref{eq:bounded-jump},
		$|x_j+\ell-x_i|\le\tau D.$

		Choose \(q\in\mathbb Z^d\) such that $|x_m+q-x_j|=d_{\mathbb T^d}(x_m,x_j),$	and set
		$\ell':=\ell+q.$ 	Then
		$(x_m+\ell'-x_i)-(x_j+\ell-x_i)=x_m+q-x_j,$
		and therefore
		\begin{equation}\label{eq:lift-difference}
			\left|
			(x_m+\ell'-x_i)-(x_j+\ell-x_i)
			\right|
			=
			d_{\mathbb T^d}(x_m,x_j).
		\end{equation}
		Moreover,
		\[
		|x_m+\ell'-x_i|
		\le
		|x_j+\ell-x_i|
		+
		d_{\mathbb T^d}(x_m,x_j)
		\le
		\tau(D+1).
		\]

		Using \((x_i,\ell')\) as a competitor in \eqref{eq:fully-discrete-ergodic-equation} at \(x_m\),
		and subtracting the equality at \(x_j\), we obtain
		\[
		u_{h,\tau}(x_m)-u_{h,\tau}(x_j)
		\le
		\widehat A_{\tau}^h(x_i,x_m,\ell')
		-
		\widehat A_{\tau}^h(x_i,x_j,\ell).
		\]
		The corresponding velocities
		$v:= \frac{x_j+\ell-x_i}{\tau},$ and $ w:= \frac{x_m+\ell'-x_i}{\tau}$
		belong to \(B_{D+1}\). Hence, by the mean value theorem and the boundedness of
		\(D_vL\) on
		\(\mathbb T^d\times\overline B_{D+1}\),
		\[
		\begin{aligned}
			&
			\widehat A_{\tau}^h(x_i,x_m,\ell')
			-
			\widehat A_{\tau}^h(x_i,x_j,\ell)
			=
			\tau\bigl(L(x_i,w)-L(x_i,v)\bigr)
			\\
			&\qquad \le
			\tau
			\sup_{\mathbb T^d\times B_{D+1}}|D_vL|
			\,|w-v|
			\\
			&\qquad =
			\sup_{\mathbb T^d\times B_{D+1}}|D_vL|
			\left|
			(x_m+\ell'-x_i)-(x_j+\ell-x_i)
			\right|.
		\end{aligned}
		\]
		Using \eqref{eq:lift-difference}, we conclude that
		$u_{h,\tau}(x_m)-u_{h,\tau}(x_j) \le C_3d_{\mathbb T^d}(x_m,x_j).$
		Exchanging the roles of \(x_j\) and \(x_m\) gives the reverse inequality.
		Therefore
		\begin{equation}\label{eq:local-lip}
			|u_{h,\tau}(x_m)-u_{h,\tau}(x_j)|
			\le
			C_3d_{\mathbb T^d}(x_m,x_j)
		\end{equation}
		whenever $d_{\mathbb T^d}(x_m,x_j)\le\tau.$

		Finally, let \(y,z\in\mathcal G_h\) be arbitrary. If
		\(d_{\mathbb T^d}(y,z)\le\tau\), then \eqref{eq:local-lip} applies. If
		instead \(d_{\mathbb T^d}(y,z)>\tau\), applying
		\eqref{eq:large-scale-lip} to \((y,z)\) and then to \((z,y)\) gives
		$|u_{h,\tau}(y)-u_{h,\tau}(z)| \le C_2d_{\mathbb T^d}(y,z).$
		In both cases we obtain
		\[
		|u_{h,\tau}(y)-u_{h,\tau}(z)|
		\le
		C\,d_{\mathbb T^d}(y,z)
		\qquad
		\forall y,z\in\mathcal G_h,
		\]
		with \(C\) independent of \(h\) and \(\tau\). Together with
		\eqref{eq:bounded-jump}, this proves the proposition.
	\end{proof}

We now compare the fully--discrete critical value directly with the continuous
	one. The   estimate adapts the long-time comparison argument of
	\cite[Theorems 3.3 and 3.7]{BouillardFaouZavidovique} to the present
	left-endpoint cost and winding representation. Let \(\mathcal T^t\) denote the backward Lax--Oleinik semigroup,
	defined for \(u\in C(\mathbb T^d)\) by
	\[
	(\mathcal T^t u)(y)
	:=
	\inf_{\substack{\gamma\in AC([0,t];\mathbb T^d)\\
			\gamma(t)=y}}
	\left\{
	u(\gamma(0))
	+
	\int_0^t L(\gamma(s),\dot\gamma(s))\,ds
	\right\}.
	\]
		For every critical solution \(u\), one has
	\[
	\mathcal T^t u=u-t\alpha(H).
	\]

	\begin{theorem}
		\label{thm:continuous-fully-discrete-critical-error}
		There exist $\tau_0,\eta_0,C>0$ such that
		\[
		\bigl|\bar L(h,\tau)+\alpha(H)\bigr|
		\le C\left(\tau+\frac h\tau\right)
		\]
		whenever $0<\tau\le\tau_0$ and $h/\tau\le\eta_0$.
	\end{theorem}

	\begin{proof}
		Let $u$ be a critical solution, set
		$\phi_h:=u|_{\mathcal G_h}$, and define
		\[
		(E_\tau u)(y)
		:=
		\inf_{\substack{x\in\mathbb T^d\\ \ell\in\mathbb Z^d}}
		\left\{
		u(x)+\tau L\left(x,\frac{y+\ell-x}{\tau}\right)
		\right\}.
		\]
		The infimum is attained, since $\mathbb T^d$ is compact and the
		uniform superlinearity of $L$ makes the cost coercive with respect
		to the winding label.

		We first establish uniform bounds for the relevant velocities. Let
		$K:=\operatorname{Lip}(u)$, let $(x,\ell)$ minimize $E_\tau u(y)$,
		and set $v:=(y+\ell-x)/\tau$. Comparison with the admissible pair
		$(y,0)$ gives
		\[
		u(x)+\tau L(x,v)
		\le
		u(y)+\tau L(y,0).
		\]
		Since $d_{\mathbb T^d}(x,y)\le\tau|v|$, it follows that
		\[
		L(x,v)-K|v|
		\le
		\max_{z\in\mathbb T^d}L(z,0).
		\]
		Uniform superlinearity therefore yields $|v|\le R$, with $R$
		independent of $y$ and $0<\tau\le\tau_0$.

		Moreover, since $\mathcal T^\tau u=u-\tau\alpha(H)$, every minimizing
		curve in $\mathcal T^\tau u(y)$ is $u$-calibrated. The standard a
		priori compactness estimate for calibrated Tonelli curves gives,
		after enlarging $R$, $|\dot\gamma(s)|\le R$ for $0\le s\le\tau$.
		The Euler--Lagrange equation reads
		\[
		D^2_{vv}L(\gamma,\dot\gamma)\ddot\gamma
		=
		D_xL(\gamma,\dot\gamma)
		-D^2_{xv}L(\gamma,\dot\gamma)\dot\gamma.
		\]
		Since $D^2_{vv}L$ is uniformly positive definite on
		$\mathbb T^d\times\overline B_R$, we also obtain
		$|\ddot\gamma(s)|\le C$.

		We now compare $E_\tau u$ with $\mathcal T^\tau u$. Let $(x,\ell)$
		minimize $E_\tau u(y)$ and consider the affine lifted curve joining a
		lift of $x$ to $y+\ell$. Its velocity is $v$, and hence
		\[
		\left|
		\int_0^\tau L(\gamma(s),v)\,ds-\tau L(x,v)
		\right|
		\le C\tau^2.
		\]
		Using this curve as a competitor gives
		$\mathcal T^\tau u(y)\le E_\tau u(y)+C\tau^2$.

		Conversely, let $\gamma$ minimize $\mathcal T^\tau u(y)$, choose a
		lift $\widetilde\gamma$, and set
		$ x:=\gamma(0)$ and
		$\bar v:=(\widetilde\gamma(\tau)-\widetilde\gamma(0))/\tau$.
		The acceleration bound gives
		\[
		|\dot{\widetilde\gamma}(s)-\bar v|
		\le
		\frac1\tau\int_0^\tau
		|\dot{\widetilde\gamma}(s)-\dot{\widetilde\gamma}(t)|\,dt
		\le C\tau,
		\]
		while $d_{\mathbb T^d}(\gamma(s),x)\le R\tau$. Therefore, by the
		$C^1$ regularity of $L$ on the relevant compact set,
		\[
		\left|
		\int_0^\tau L(\gamma(s),\dot\gamma(s))\,ds
		-\tau L(x,\bar v)
		\right|
		\le C\tau^2.
		\]
		Writing $\widetilde\gamma(\tau)=y+\ell$ for some
		$\ell\in\mathbb Z^d$, after choosing the initial lift consistently,
		we have $\bar v=(y+\ell-x)/\tau$. Thus $(x,\ell)$ is admissible in
		the definition of $E_\tau u(y)$, and
		$E_\tau u(y)\le\mathcal T^\tau u(y)+C\tau^2$. Consequently,
		\begin{equation}
		\label{eq:euler-exact-one-step}
		\|E_\tau u-\mathcal T^\tau u\|_\infty
		\le C\tau^2.
		\end{equation}

		We next estimate the spatial discretization error. For
		$y\in\mathcal G_h$, restriction of the initial point gives
		$T_{h,\tau}\phi_h(y)\ge E_\tau u(y)$. Let $(x,\ell)$ minimize
		$E_\tau u(y)$, and choose lifts $\widetilde x$ and
		$\widetilde y=y+\ell$. Select $x_i\in\mathcal G_h$ and a lift
		$\widetilde x_i$ such that
		$|\widetilde x_i-\widetilde x|\le C_dh$. Keeping the terminal lift
		$\widetilde y$, set
		\[
		v_i:=\frac{\widetilde y-\widetilde x_i}{\tau}.
		\]
		The corresponding displacement determines an admissible winding label,
		and $|v_i-v|\le C_dh/\tau$. Since $h/\tau\le\eta_0$, the two
		velocities belong to a fixed compact set. Hence
		$|u(x_i)-u(x)|\le Kh$ and, after choosing $\tau_0\le1$,
		\[
		\tau|L(x_i,v_i)-L(x,v)|
		\le
		C\tau\left(h+\frac h\tau\right)
		\le Ch.
		\]
		Therefore $T_{h,\tau}\phi_h(y)\le E_\tau u(y)+Ch$, and thus
		\begin{equation}
		\label{eq:space-one-step}
		0\le
		T_{h,\tau}\phi_h-(E_\tau u)|_{\mathcal G_h}
		\le Ch.
		\end{equation}

		Combining \eqref{eq:euler-exact-one-step} and
		\eqref{eq:space-one-step} with
		$\mathcal T^\tau u=u-\tau\alpha(H)$, we obtain
		\begin{equation}
		\label{eq:critical-one-step-residual}
		\|T_{h,\tau}\phi_h-\phi_h+\tau\alpha(H)\|_\infty
		\le C(\tau^2+h).
		\end{equation}

		The operator $T_{h,\tau}$ is additively homogeneous and nonexpansive
		in the uniform norm. Iterating
		\eqref{eq:critical-one-step-residual} therefore gives
		\[
		\|T_{h,\tau}^N\phi_h-\phi_h+N\tau\alpha(H)\|_\infty
		\le CN(\tau^2+h).
		\]
		Let $u_{h,\tau}$ be a fully--discrete critical solution. Then
		\[
		T_{h,\tau}^Nu_{h,\tau}
		=
		u_{h,\tau}+N\tau\bar L(h,\tau).
		\]
		By nonexpansiveness,
		\[
		\left\|
		T_{h,\tau}^N\phi_h-u_{h,\tau}
		-N\tau\bar L(h,\tau)
		\right\|_\infty
		\le
		\|\phi_h-u_{h,\tau}\|_\infty.
		\]
		Combining the last two estimates yields
		\[
		N\tau\bigl|\bar L(h,\tau)+\alpha(H)\bigr|
		\le
		CN(\tau^2+h)+2\|\phi_h-u_{h,\tau}\|_\infty.
		\]
		For fixed $h,\tau$, division by $N\tau$ and passage to the limit as
		$N\to\infty$ give
		\[
		\bigl|\bar L(h,\tau)+\alpha(H)\bigr|
		\le
		C\left(\tau+\frac h\tau\right),
		\]
		which proves the claim.
	\end{proof}

	In particular, under
	\begin{equation}\label{eq:conv_parameter}
	\tau_n\to0,\qquad h_n\to0,\qquad h_n/\tau_n\to0,
	\end{equation}
	one has $\bar L(h_n,\tau_n)\to-\alpha(H)$.

The uniform estimates give compactness of the normalized discrete
solutions. The consistency of the scheme then identifies every
accumulation point as a critical viscosity solution.
\begin{proposition}
	\label{prop:convergence-critical-solutions}
	Assume \eqref{eq:conv_parameter}. Let \(u_{h_n,\tau_n}\) be normalized solutions of   \eqref{eq:fully-discrete-ergodic-equation} and let
	\(\widetilde u_{h_n,\tau_n}\) denote their piecewise affine periodic
	interpolations on \(\mathbb T^d\). Then there exist a subsequence, not
	relabeled, and a Lipschitz continuous function
	\(u\colon\mathbb T^d\to\mathbb R\) such that
	\[
	\widetilde u_{h_n,\tau_n}\longrightarrow u
	\qquad\text{uniformly on }\mathbb T^d.
	\]
	Moreover, \(u\) is a viscosity solution of the critical Hamilton--Jacobi
	equation
	\[
	H(x,Du)=\alpha(H)
	\qquad\text{in }\mathbb T^d.
	\]
\end{proposition}

\begin{proof}
	The uniform Lipschitz estimate and the normalization imply that
	\((\widetilde u_{h_n,\tau_n})_n\) is precompact in
	\(C(\mathbb T^d)\). Hence, up to a subsequence,
	\(\widetilde u_{h_n,\tau_n}\to u\) uniformly for some Lipschitz
	function \(u\).

	It remains to identify the limit. Let \(\phi\in C^2(\mathbb T^d)\)
	and let \(x\) be a strict local maximum of \(u-\phi\). Choosing grid
	points \(x_n\to x\) where
	\(u_{h_n,\tau_n}-\phi\) attains a corresponding discrete maximum and
	using the fully--discrete critical equation, we obtain, by consistency,
	$H(x,D\phi(x))\le\alpha(H).$
	Here the consistency error tends to zero because
	\(\tau_n\to0\) and \(h_n/\tau_n\to0\), while
	\(\bar L(h_n,\tau_n)\to-\alpha(H)\). The analogous argument at a
	strict local minimum gives
	$H(x,D\phi(x))\ge\alpha(H).$
	Thus \(u\) is both a viscosity subsolution and supersolution of the
	critical Hamilton--Jacobi equation.
\end{proof}


	\section{The Fully--Discrete Mather Set}

	In this section, we define the fully--discrete Mather set associated with the
	fully--discrete Lax--Oleinik scheme. The construction is
	based on probability measures on discrete transitions
	$(x_i,x_j,\ell)\in\mathcal G_h\times\mathcal G_h\times\mathbb Z^d$,
	where the integer vector $\ell$ records the winding of the transition.

	\begin{definition}
		A family
		$\mu=\left(\mu_{ij}^{\ell}\right)_
		{x_i,x_j\in\mathcal G_h,\ \ell\in\mathbb Z^d}$
		is called a fully--discrete holonomic probability measure if
		$\mu_{ij}^{\ell}\ge0$ and $\sum_{i,j,\ell}\mu_{ij}^{\ell}=1$,
		and
		\begin{equation}\label{eq:fd-holonomy-winding}
			\sum_{j,\ell}\mu_{ij}^{\ell}
			=
			\sum_{j,\ell}\mu_{ji}^{\ell},
			\qquad
			\forall x_i\in\mathcal G_h.
		\end{equation}
		We denote by $\mathcal H_{h,\tau}$ the set of fully--discrete holonomic
		probability measures.
	\end{definition}

	Equivalently, \eqref{eq:fd-holonomy-winding} can be written in weak form as
	\begin{equation}\label{eq:fd-holonomy-weak-winding}
		\sum_{i,j,\ell}
		\bigl(\varphi(x_j)-\varphi(x_i)\bigr)\mu_{ij}^{\ell}=0,
		\qquad
		\forall \varphi:\mathcal G_h\to\mathbb R.
	\end{equation}
	\begin{remark}
The coefficient $\mu_{ij}^{\ell}$ is the mass assigned to the transition
$(x_i,x_j,\ell)$. Thus \eqref{eq:fd-holonomy-winding} equates incoming and
outgoing masses at every node. Since $x_j=x_i+\tau v_{ij}^{\ell}$ on
$\mathbb T^d$, its weak form \eqref{eq:fd-holonomy-weak-winding} is the
discrete counterpart of the closedness condition in phase space.
\end{remark}

	\begin{definition}
		A fully--discrete Mather measure is a minimizer of
		\begin{equation}\label{eq:fd-Mather-LP-winding}
			\min_{\mu\in\mathcal H_{h,\tau}}
			\sum_{i,j,\ell}
			L\left(x_i,\frac{x_j+\ell-x_i}{\tau}\right)\mu_{ij}^{\ell}.
		\end{equation}
		The fully--discrete Mather set is the subset $\widetilde{\mathcal M}_{h,\tau}\subset\mathbb T^d\times\mathbb R^d$ defined by
		\[ \widetilde{\mathcal M}_{h,\tau} := \overline{ \bigcup \left\{ (x_i,v_{ij}^{\ell}) : \mu_{ij}^{\ell}>0 \right\} }, \]
		where $v_{ij}^{\ell}:=(x_j+\ell-x_i)/\tau$,
		and the union is taken over all fully--discrete Mather measures. The
		minimum in \eqref{eq:fd-Mather-LP-winding} is attained, see Remark
		\ref{rem:fd-mather-existence} below.
	\end{definition}

	We reconstruct a fully--discrete holonomic measure as a probability measure on
	phase space by setting
	\begin{equation}\label{eq:fd-reconstruction-winding}
		\mathcal R_{h,\tau}\mu
		:=
		\sum_{i,j,\ell}
		\mu_{ij}^{\ell}\,
		\delta_{\left(x_i,\frac{x_j+\ell-x_i}{\tau}\right)},
	\end{equation}
	and we denote the fully--discrete action of $\mu\in\mathcal H_{h,\tau}$ by
	\begin{equation}\label{eq:fd-action-functional}
		\mathcal A_{h,\tau}(\mu)
		:=
		\sum_{i,j,\ell}L\bigl(x_i,v_{ij}^{\ell}\bigr)\mu_{ij}^{\ell}
		=
		\int_{\mathbb T^d\times\mathbb R^d}L\,d\bigl(\mathcal R_{h,\tau}\mu\bigr).
	\end{equation}

	\begin{proposition}
		\label{prop:fd-critical-duality}
		For every \(h,\tau>0\),
		\[
		\bar L(h,\tau)
		=
		\min_{\mu\in\mathcal H_{h,\tau}}
		\sum_{i,j,\ell}
		L\left(x_i,\frac{x_j+\ell-x_i}{\tau}\right)
		\mu_{ij}^{\ell}.
		\]
	\end{proposition}
	\begin{proof}
		Let \(u_{h,\tau}\) be a solution of \eqref{eq:fully-discrete-ergodic-equation}.
		Then, for every
		\(x_i,x_j\in\mathcal G_h\) and \(\ell\in\mathbb Z^d\),
		\[
		u_{h,\tau}(x_j)+\tau\bar L(h,\tau)
		\leq
		u_{h,\tau}(x_i)
		+
		\widehat A_{\tau}^h(x_i,x_j,\ell).
		\]
		Let \(\mu\in\mathcal H_{h,\tau}\). Multiplying the previous inequality by
		\(\mu_{ij}^{\ell}\) and summing over \(i,j,\ell\), we obtain
		\[
		\tau\bar L(h,\tau)
		\leq
		\tau
		\sum_{i,j,\ell}
		L\left(
		x_i,\frac{x_j+\ell-x_i}{\tau}
		\right)\mu_{ij}^{\ell}
		+
		\sum_{i,j,\ell}
		\bigl(
		u_{h,\tau}(x_i)-u_{h,\tau}(x_j)
		\bigr)\mu_{ij}^{\ell}.
		\]
		The last term vanishes by the holonomy condition. Hence
		\[
		\bar L(h,\tau)
		\leq
		\sum_{i,j,\ell}
		L\left(
		x_i,\frac{x_j+\ell-x_i}{\tau}
		\right)\mu_{ij}^{\ell}.
		\]
		Since \(\mu\in\mathcal H_{h,\tau}\) is arbitrary, it follows that
		\begin{equation}\label{eq:critical-value-lower-action}
			\bar L(h,\tau)
			\leq
			\min_{\mu\in\mathcal H_{h,\tau}}
			\sum_{i,j,\ell}
			L\left(
			x_i,\frac{x_j+\ell-x_i}{\tau}
			\right)\mu_{ij}^{\ell}.
		\end{equation}

		We prove the reverse inequality by constructing a holonomic measure supported
		on calibrated transitions. For every \(x_j\in\mathcal G_h\), choose a
		minimizing pair \((x_i,\ell)\) in \eqref{eq:fully-discrete-ergodic-equation}, so that
		\[
		u_{h,\tau}(x_j)+\tau\bar L(h,\tau)
		=
		u_{h,\tau}(x_i)
		+
		\widehat A_{\tau}^h(x_i,x_j,\ell).
		\]
		Starting from an arbitrary grid point and repeatedly choosing a minimizing
		predecessor produces a sequence of grid points. Since \(\mathcal G_h\) is
		finite, this sequence eventually contains a cycle. Thus there exist
		\(x_0,\ldots,x_{q-1}\in\mathcal G_h\) and
		\(\ell_0,\ldots,\ell_{q-1}\in\mathbb Z^d\), with indices understood modulo
		\(q\), such that
		\begin{equation}\label{eq:calibrated-cycle}
			u_{h,\tau}(x_r)+\tau\bar L(h,\tau)
			=
			u_{h,\tau}(x_{r+1})
			+
			\widehat A_{\tau}^h(x_{r+1},x_r,\ell_r),
			\qquad r=0,\ldots,q-1.
		\end{equation}

		Define the empirical measure of the cycle by
		$\mu_{ij}^{\ell}:=\frac1q\#\{r:(x_i,x_j,\ell)=(x_{r+1},x_r,\ell_r)\}.$
		It is a probability measure, and the cyclic indexing gives equal incoming and
		outgoing masses at every node; hence $\mu\in\mathcal H_{h,\tau}$.
		Summing \eqref{eq:calibrated-cycle} over \(r\), the terms involving
		\(u_{h,\tau}\) telescope, and we obtain
		\[
		q\tau\bar L(h,\tau)
		=
		\tau
		\sum_{r=0}^{q-1}
		L\left(
		x_{r+1},
		\frac{x_r+\ell_r-x_{r+1}}{\tau}
		\right).
		\]
		Dividing by \(q\tau\) gives
		\[
		\bar L(h,\tau)
		=
		\sum_{i,j,\ell}
		L\left(
		x_i,\frac{x_j+\ell-x_i}{\tau}
		\right)\mu_{ij}^{\ell}.
		\]
		Consequently,
		\[
		\min_{\mu\in\mathcal H_{h,\tau}}
		\sum_{i,j,\ell}
		L\left(
		x_i,\frac{x_j+\ell-x_i}{\tau}
		\right)\mu_{ij}^{\ell}
		\leq
		\bar L(h,\tau).
		\]
		Combining this inequality with
		\eqref{eq:critical-value-lower-action} proves the result.
	\end{proof}

	\begin{remark} \label{rem:fd-mather-existence} 
		The cycle measure constructed in the proof of Proposition \ref{prop:fd-critical-duality} attains the infimum in \eqref{eq:fd-Mather-LP-winding}; hence fully--discrete Mather measures exist for every \(h,\tau>0\). 
		For any fully--discrete Mather measure, the Bellman inequality is an equality on every transition in its support. Indeed, holonomy cancels the increments of a critical solution, while minimality makes the integral of the nonnegative Bellman residual equal to zero. Thus every transition in the support realizes the minimum in \eqref{eq:fully-discrete-ergodic-equation}. 
		Proposition \ref{prop:fd-lip-bound} bounds its velocity by \(D\), so only finitely many transitions may carry positive mass. Therefore \(\widetilde{\mathcal M}_{h,\tau}\) is finite and the closure in its definition is redundant. \end{remark}
	\begin{remark}
		The winding labels retain the homotopy information of the transitions. In particular, one may associate with a fully--discrete holonomic measure $\mu$ the rotation vector
		\[ \rho_{h,\tau}(\mu) := \sum_{i,j,\ell}v_{ij}^{\ell}\mu_{ij}^{\ell} = \frac{1}{\tau} \sum_{i,j,\ell}\ell\,\mu_{ij}^{\ell}, \]
		where the second identity follows from the holonomy condition. This makes it natural to introduce fully--discrete analogues of Mather's $\alpha$ and $\beta$ functions (see \cite{Mather1991}) by minimizing the discrete action under prescribed cohomology or rotation constraints. We do not develop this extension here, but it shows that the fully--discrete theory also has an intrinsic variational meaning.
	\end{remark}
	\section{Convergence of Fully--Discrete Mather Measures and Sets}
	\label{sec:convergence-fd-mather}
	We now study the limit $h,\tau\to0$ of fully--discrete Mather measures and sets.
	\begin{proposition}
		\label{prop:fd-mather-measures-to-continuous}
		Assume \eqref{eq:conv_parameter}. 	Let \(\mu^{h_n,\tau_n}\in\mathcal H_{h_n,\tau_n}\) be fully--discrete Mather
		measures and set
		\[
		\widetilde\mu^{h_n,\tau_n}
		:=
		\mathcal R_{h_n,\tau_n}\mu^{h_n,\tau_n}.
		\]
		Then the family \((\widetilde\mu^{h_n,\tau_n})_n\) is tight in
		\(\mathcal P(\mathbb T^d\times\mathbb R^d)\). Moreover, every narrow
		accumulation point is a continuous Mather measure.
	\end{proposition}

	\begin{proof}
		Set
		$\widetilde\mu_n := \mathcal R_{h_n,\tau_n}\mu^{h_n,\tau_n}.$
		Since the diagonal transition \(x_i\to x_i\), with \(\ell=0\), is
		admissible, the minimizing property gives
		$\int L(x,v)\,d\widetilde\mu_n(x,v) \le \max_{x\in\mathbb T^d}L(x,0) =:C.$

		Since \(L\) is uniformly superlinear in \(v\) and bounded from below,
		there exist a nondecreasing superlinear function
		\(\Theta:[0,+\infty)\to[0,+\infty)\) and a constant \(C_0\) such that
		\[
		\Theta(|v|)\le L(x,v)+C_0
		\qquad
		\text{for all }(x,v)\in\mathbb T^d\times\mathbb R^d.
		\]
		Consequently,
		\begin{equation}
			\label{eq:uniform-integrability-velocities}
			\sup_n\int\Theta(|v|)\,d\widetilde\mu_n(x,v)<+\infty.
		\end{equation}
		In particular, the velocities are uniformly integrable:
		$\lim_{R\to\infty} \sup_n \int_{\{|v|>R\}}|v|\,d\widetilde\mu_n=0.$
		Since \(\mathbb T^d\) is compact, this also implies tightness. Thus,
		up to a subsequence,
		\[
		\widetilde\mu_n\rightharpoonup\mu
		\qquad\text{narrowly in }
		\mathcal P(\mathbb T^d\times\mathbb R^d).
		\]
		By lower semicontinuity, \(\int\Theta(|v|)\,d\mu<+\infty\), so the same
		tail control holds for \(\mu\).

		We prove that \(\mu\) is closed. For
		\(\varphi\in C^1(\mathbb T^d)\), discrete holonomy gives
		\[
		\int F_n(x,v)\,d\widetilde\mu_n(x,v)=0,
		\qquad
		F_n(x,v):=
		\int_0^1D\varphi(x+s\tau_n v)\cdot v\,ds.
		\]
		For every fixed \(R>0\),
		$F_n(x,v)\longrightarrow D\varphi(x)\cdot v$
		uniformly on \(\mathbb T^d\times\overline B_R\). Hence narrow
		convergence yields
		\[
		\lim_{n\to\infty}
		\int_{\{|v|\le R\}}F_n\,d\widetilde\mu_n
		=
		\int_{\{|v|\le R\}}D\varphi(x)\cdot v\,d\mu,
		\]
		using, if necessary, a continuous cutoff in place of the
		characteristic function of \(\{|v|\le R\}\).

		Moreover,
		$|F_n(x,v)| \le \|D\varphi\|_\infty |v|,$
		so \eqref{eq:uniform-integrability-velocities} allows us to let
		\(R\to\infty\), uniformly in \(n\). We conclude that
		$\int_{\mathbb T^d\times\mathbb R^d} D\varphi(x)\cdot v\,d\mu(x,v)=0,$
		and therefore \(\mu\) is closed.
		Finally, by Proposition \ref{prop:fd-critical-duality},
		\[
		\int L(x,v)\,d\widetilde\mu_n(x,v)
		=
		\bar L(h_n,\tau_n)
		\longrightarrow-\alpha(H).
		\]
		Since \(L\) is continuous and bounded from below, lower
		semicontinuity gives
		\[
		\int L(x,v)\,d\mu(x,v)
		\le
		\liminf_{n\to\infty}
		\int L(x,v)\,d\widetilde\mu_n(x,v)
		=
		-\alpha(H).
		\]
		On the other hand, \(\mu\) is closed, and hence
		$\int L(x,v)\,d\mu(x,v)\ge-\alpha(H).$
		Thus equality holds, and \(\mu\) is a continuous Mather measure.
	\end{proof}

	\subsection{The Kuratowski upper limit}

	We now study the Kuratowski limits of the fully--discrete Mather sets. Recall that, for a sequence $\{A_k\}_{k\in\mathbb N}$ in a metric space,
	\[
	x\in\limsup_{k\to\infty}A_k \quad\Longleftrightarrow\quad \exists\,k_j\to\infty,\ \exists\,x_{k_j}\in A_{k_j} \text{ such that }x_{k_j}\to x,
	\]
	and
	\[
	x\in\liminf_{k\to\infty}A_k \quad\Longleftrightarrow\quad \exists\,x_k\in A_k \text{ such that }x_k\to x.
	\]
	Kuratowski convergence to $A$ means that both limits coincide with $A$. \par
	We first consider the upper limit. Although Proposition \ref{prop:fd-mather-measures-to-continuous} gives convergence of minimizing measures, narrow convergence does not control their supports from above. Accordingly, the calibration argument below gives an upper inclusion into the continuous Ma\~n\'e set, rather than into the Aubry or Mather set.
	 We   introduce the strengthened scaling
	\begin{equation}\label{eq:conv_parameter_strong}
		\tau_n\longrightarrow0,
		\qquad
		h_n\le C_\ast\,\tau_n^{2}
		\quad\text{for some constant }C_\ast>0 .
	\end{equation}
	Since $h_n/\tau_n\le C_\ast\tau_n\to0$, condition
	\eqref{eq:conv_parameter_strong} implies \eqref{eq:conv_parameter}. As
	explained in Appendix \ref{app:technical-lemma} and in Section
	\ref{sec:implementation}, the quadratic regime is precisely the one in which
	the spatial contribution to the discrete Euler--Lagrange estimate does not
	dominate the time-stepping one
	\begin{proposition}
		\label{prop:fd-mather-limsup-mane}
		Under \eqref{eq:conv_parameter_strong},
		\[
		\limsup_{n\to\infty}\widetilde{\mathcal M}_{h_n,\tau_n}
		\subset\widetilde{\mathcal N}_L.
		\]
	\end{proposition}

	\begin{proof}
		By the definition of the Kuratowski upper limit, it suffices to consider a
		subsequence, not relabeled, and points
		$(x_n,v_n)\in\widetilde{\mathcal M}_{h_n,\tau_n}$ converging to
		$(x,v)$. For every \(n\), there exists a fully--discrete Mather
		measure \(\mu_n\) whose support contains a transition
		$(x_n,x_n^+,\ell_n)$ satisfying
		$v_n=\frac{x_n^++\ell_n-x_n}{\tau_n}.$

		Let \(u_n:=u_{h_n,\tau_n}\) be normalized fully--discrete critical
		solutions and set \(\bar L_n:=\bar L(h_n,\tau_n)\). Since
		\eqref{eq:conv_parameter_strong} implies \eqref{eq:conv_parameter}, we may
		apply Proposition \ref{prop:convergence-critical-solutions} and extract a
		further subsequence, again not relabeled, along which the piecewise affine
		interpolations \(\widetilde u_n\) converge uniformly on \(\mathbb T^d\)
		to a critical viscosity solution \(u\).  Define the Bellman defect by
		\[
		G_n(x_i,x_j,\ell)
		:=
		u_n(x_i)
		+\tau_nL\left(x_i,\frac{x_j+\ell-x_i}{\tau_n}\right)
		-u_n(x_j)-\tau_n\bar L_n.
		\]
		The Bellman inequality gives \(G_n\ge0\). Since \(\mu_n\) is holonomic
		and minimizing, Proposition \ref{prop:fd-critical-duality} yields
		\[
		\sum_{i,j,\ell}G_n(x_i,x_j,\ell)\mu_{n,ij}^{\ell}
		=
		\tau_n\bigl(\mathcal A_{h_n,\tau_n}(\mu_n)-\bar L_n\bigr)
		=0.
		\]
		Therefore \(G_n=0\) on \(\supp(\mu_n)\).

		The holonomy condition implies that every positive-mass edge can be
		extended both forward and backward through positive-mass edges.
		Hence $(x_n,x_n^+,\ell_n)$ belongs to a bi-infinite chain
		$(x_k^n,x_{k+1}^n,\ell_k^n)_{k\in\mathbb Z}$ contained in
		\(\supp(\mu_n)\), with
		\[
		x_0^n=x_n,\qquad x_1^n=x_n^+,\qquad \ell_0^n=\ell_n.
		\]
		Since \(G_n\) vanishes on the support, every edge of this chain is
		calibrated. Moreover, its initial velocity is
		$v_0^n = \frac{x_1^n+\ell_0^n-x_0^n}{\tau_n} = v_n.$
		Thus $(x_0^n,v_0^n)\to(x,v)$. Lemma
		\ref{lem:compactness-pointed-calibrated-chains}, applied along the
		subsequence selected above, provides an Euler--Lagrange trajectory
		\(\gamma\) which is $(u,L,\alpha(H))$--calibrated on $\mathbb R$ and
		satisfies \(\gamma(0)=x\) and \(\dot\gamma(0)=v\). Consequently, by \eqref{eq:aubry-mane-sets},
		$(x,v)\in\widetilde{\mathcal I}(u)\subset\widetilde{\mathcal N}_L$,
		which proves the claim.
	\end{proof}

	\begin{remark}
		The argument yields the Ma\~n\'e set since the limiting trajectory is
		calibrated by one critical solution $u$. Membership in the Aubry set would
		require the stronger {\it staticity property}, equivalently calibration by every backward weak KAM solution. This does not follow from the compactness argument above.
	\end{remark}

	\subsection{The Kuratowski lower limit}

	The lower inclusion requires approximating every point of the continuous Mather set by supports of fully--discrete minimizing measures.

	\begin{proposition}
		\label{prop:fd-mather-liminf-unique}
		Assume \eqref{eq:conv_parameter} and that the continuous Mather
		measure \(\mu\) is unique. Then
		\[
		\supp(\mu)=\widetilde{\mathcal M}_L
		\subset
		\liminf_{n\to\infty}\widetilde{\mathcal M}_{h_n,\tau_n}.
		\]
	\end{proposition}

	\begin{proof}
		For each \(n\), let \(\mu^{h_n,\tau_n}\) be a fully--discrete Mather
		measure and set
		$\widetilde\mu_n := \mathcal R_{h_n,\tau_n}\mu^{h_n,\tau_n}.$
		By Proposition \ref{prop:fd-mather-measures-to-continuous}, every
		accumulation point of \((\widetilde\mu_n)_n\) is a continuous Mather
		measure. Since the latter is unique, the whole sequence converges
		narrowly to \(\mu\).

		Let \(z\in\supp(\mu)\). For every \(m\ge1\),
		\(\mu(B_{1/m}(z))>0\); hence, by the Portmanteau theorem, there exists
		\(N_m\) such that
		\[
		\supp(\widetilde\mu_n)\cap B_{1/m}(z)\neq\emptyset
		\qquad\text{for every }n\ge N_m.
		\]
		We may choose \((N_m)_m\) strictly increasing. For
		\(N_m\le n<N_{m+1}\), choose
		$z_n\in\supp(\widetilde\mu_n)\cap B_{1/m}(z).$
		After choosing the finitely many remaining terms arbitrarily, we
		obtain \(z_n\to z\). Since
		\[
		\supp(\widetilde\mu_n)
		\subset\widetilde{\mathcal M}_{h_n,\tau_n},
		\]
		it follows that
		\(z\in\liminf_{n\to\infty}\widetilde{\mathcal M}_{h_n,\tau_n}\).
		The conclusion follows because \(z\in\supp(\mu)\) was arbitrary.
	\end{proof}

	Combining the upper and lower Kuratowski inclusions, we obtain full convergence whenever the continuous Ma\~n\'e and Mather sets coincide.
	\begin{corollary} \label{cor:fd-mather-kuratowski-convergence}
		Assume \eqref{eq:conv_parameter_strong}, that the continuous Mather measure is unique, and that
		$\widetilde{\mathcal N}_L=\widetilde{\mathcal M}_L$. 
		Then
		\[ \widetilde{\mathcal M}_{h_n,\tau_n} \xrightarrow{K} \widetilde{\mathcal M}_L \qquad\text{as }n\to\infty. \]
	\end{corollary}
\begin{proof} By Propositions \ref{prop:fd-mather-limsup-mane} and \ref{prop:fd-mather-liminf-unique}, 
		\[
		 \limsup_{n\to\infty}\widetilde{\mathcal M}_{h_n,\tau_n} \subset \widetilde{\mathcal N}_L = \widetilde{\mathcal M}_L \subset \liminf_{n\to\infty}\widetilde{\mathcal M}_{h_n,\tau_n}. 
		\]
 Since the lower Kuratowski limit is contained in the upper one, all the sets above coincide. \end{proof}

\begin{remark}
	The assumptions of Corollary
	\ref{cor:fd-mather-kuratowski-convergence} hold, for instance, for
	\[
	L_\omega(x,v):=\frac12|v-\omega|^2,
	\]
	where \(\omega\in\mathbb R^d\) is nonresonant, that is, \(k\cdot\omega\neq0\) for every \(k\in\mathbb Z^d\setminus\{0\}\). The associated
	Kronecker flow is uniquely ergodic, and the unique Mather measure is $\mathcal L_{\mathbb T^d}^d\otimes\delta_\omega$.	Moreover, every semistatic curve satisfies
	\(\dot\gamma=\omega\); hence
	\[
	\widetilde{\mathcal N}_{L_\omega}
	=
	\widetilde{\mathcal M}_{L_\omega}
	=
	\mathbb T^d\times\{\omega\}.
	\]
\end{remark}


\begin{remark}
	\label{rem:genericity-of-the-assumptions}
	The two hypotheses of Corollary
	\ref{cor:fd-mather-kuratowski-convergence} are less restrictive than they
	may appear. First, they reduce to a single one: since each static class
	supports at least one ergodic minimizing measure, and a semistatic orbit
	which is not static connects two \emph{distinct} static classes
	\cite{ContrerasDelgadoIturriaga,Bernard2002}, uniqueness of the Mather
	measure already forces
	$\widetilde{\mathcal A}_L=\widetilde{\mathcal N}_L$; what is really assumed
	besides uniqueness is that the Mather measure have full support in the
	Aubry set.

	Second, this situation is generic. Since
	$H^1(\mathbb T^d;\mathbb R)\simeq\mathbb R^d$, each class is represented by
	a constant $1$--form, $L_c(x,v):=L(x,v)-c\cdot v$ is again Tonelli, and all
	our results apply verbatim to $L_c$; the winding labels make the class
	exactly representable on the grid, since
	$\tau L_c(x_i,v_{ij}^\ell)=\tau L(x_i,v_{ij}^\ell)-c\cdot(x_j+\ell-x_i)$,
	so that no further discretization error is introduced. Recall now that a
	property is generic in the sense of Ma\~n\'e if, for every Tonelli
	Lagrangian $L$, it holds for $L-u$ for all $u$ in a residual subset of
	$C^\infty(\mathbb T^d)$; potentials are precisely the perturbations that a
	scheme such as ours sees. Ma\~n\'e proved that, for each fixed $c$,
	generically the minimizing measure is unique and uniquely ergodic
	\cite{Mane1996}, and Bernard and Contreras that a single residual set of
	potentials bounds, simultaneously for all $c$, the number of ergodic
	minimizing measures and of static classes by $1+d$
	\cite{BernardContreras2008}. Building on these results, Zhang
	\cite{Zhang2017} showed that for a generic Tonelli Lagrangian there is a
	residual set $\mathcal G^\ast\subset H^1(\mathbb T^d;\mathbb R)$ with
	\[
	\widetilde{\mathcal M}_{L_c}
	=\widetilde{\mathcal A}_{L_c}
	=\widetilde{\mathcal N}_{L_c},
	\qquad
	c\in\mathcal G^\ast,
	\]
	and $\widetilde{\mathcal M}_{L_c}$ uniquely ergodic: exactly the hypotheses
	of Corollary \ref{cor:fd-mather-kuratowski-convergence}. Thus, up to an
	arbitrarily small perturbation of the potential, Kuratowski convergence of
	the fully--discrete Mather sets holds for a dense set of cohomology
	classes, the assumption ruling out only a small family of degenerate
	configurations. Such degeneracies do occur --- for the pendulum at the
	boundary of a flat of $\alpha$, or for the Denjoy sets of
	\cite{Arnaud2011}, one has
	$\widetilde{\mathcal M}_{L_c}\subsetneq\widetilde{\mathcal A}_{L_c}$ --- and
	there Proposition \ref{prop:fd-mather-limsup-mane} only yields the upper
	inclusion into the Ma\~n\'e set. 
\end{remark}


	\section{Near-Minimizing Mass-Threshold Approximations}
	\label{sec:near-minimizing-mass-threshold}

	The fully--discrete Mather set is defined through exact minimizers of the
	discrete action. Although this preserves the intrinsic variational structure
	of the discrete problem, it may fail to approximate the whole continuous
	Mather set. Indeed, when the continuous problem has several minimizing
	components, the spatial discretization may favor only some of them, excluding
	components that are asymptotically minimizing but never exact discrete
	minimizers.

	To reduce this selection effect, we introduce an approximation based on an
	action tolerance $\varepsilon\ge0$ and a local mass threshold $\delta>0$.
	The tolerance retains almost-minimizing components, while the mass threshold
	excludes regions carrying negligible mass.
	This construction is intended as
	an approximation device rather than as an alternative fully--discrete Mather
	theory.

	Let $X:=\mathbb T^d\times\mathbb R^d$, endowed with a product metric $d_X$,
	and recall the fully--discrete action $\mathcal A_{h,\tau}$ introduced in
	\eqref{eq:fd-action-functional}. By Proposition
	\ref{prop:fd-critical-duality},
	$\min_{\mu\in\mathcal H_{h,\tau}}\mathcal A_{h,\tau}(\mu)
	=\bar L(h,\tau)$.

	\begin{definition}
		\label{def:fd-near-minimizing-measures}
		For $\varepsilon\ge0$, define
		\[
		\mathcal H_{h,\tau}^{\varepsilon}
		:=
		\left\{
		\mu\in\mathcal H_{h,\tau}:
		\mathcal A_{h,\tau}(\mu)
		\le
		\bar L(h,\tau)+\varepsilon
		\right\},
		\]
		and let
		$	\widetilde{\mathcal H}_{h,\tau}^{\varepsilon}
		:=
		\left\{
		\mathcal R_{h,\tau}\mu:
		\mu\in\mathcal H_{h,\tau}^{\varepsilon}
		\right\}
		\subset\mathcal P(X)$.
	\end{definition}

	For $\varepsilon=0$, this is the family of reconstructed fully--discrete
	Mather measures. Positive values of $\varepsilon$ also retain components whose
	action differs from the discrete minimum by at most $\varepsilon$.

	Almost-minimality alone, however, does not control supports. Indeed, an $\varepsilon$--minimizing holonomic measure may assign an arbitrarily small amount of mass to nonminimizing transitions, which would nevertheless belong to its support. A positive local mass threshold is therefore needed to exclude such vanishing-mass components. We use a regularized local mass, defined through a continuous kernel, in order to ensure stability with respect to the phase-space point and narrow convergence of measures.
	Let $\kappa:[0,+\infty)\to[0,1]$ be nonincreasing and Lipschitz, with
	$\kappa=1$ on $[0,1/2]$ and $\kappa=0$ on $[1,+\infty)$. For
	$\nu\in\mathcal P(X)$, $r>0$, and $z\in X$, define
	\begin{equation}
		\label{eq:near-min-regularized-local-mass}
		m_{\nu,r}(z)
		:=
		\int_X
		\kappa\left(\frac{d_X(z,w)}r\right)
		\,d\nu(w).
	\end{equation}
	Then
	\begin{equation}
		\label{eq:near-min-regularized-local-mass-bounds}
		\nu(B_{r/2}(z))
		\le
		m_{\nu,r}(z)
		\le
		\nu(B_r(z)).
	\end{equation}

	\begin{definition}
		\label{def:near-minimal-local-mass}
		For $h,\tau,r>0$ and $\varepsilon\ge0$, set
		\[
		Q_{h,\tau,r}^{\varepsilon}(z)
		:=
		\sup_{\widetilde\mu\in
			\widetilde{\mathcal H}_{h,\tau}^{\varepsilon}}
		m_{\widetilde\mu,r}(z).
		\]
		For $0<\delta\le1$, define the near-minimizing mass-threshold approximation
		\begin{equation}
			\label{eq:near-minimizing-mass-threshold-set}
			\mathcal S_{h,\tau}^{\varepsilon,r,\delta}
			:=
			\left\{
			z\in X:
			Q_{h,\tau,r}^{\varepsilon}(z)\ge\delta
			\right\}.
		\end{equation}
	\end{definition}

	Thus, $z\in\mathcal S_{h,\tau}^{\varepsilon,r,\delta}$ if some
	fully--discrete holonomic measure has action gap at most $\varepsilon$ and
	assigns regularized mass at least $\delta$ near $z$. Equivalently, define the local action gap by
	\[
	\Gamma_{h,\tau}^{r,\delta}(z)
	:=
	\inf
	\left\{
	\mathcal A_{h,\tau}(\mu)-\bar L(h,\tau):
	\mu\in\mathcal H_{h,\tau},\
	m_{\mathcal R_{h,\tau}\mu,r}(z)\ge\delta
	\right\},
	\]
	with the convention that the infimum is $+\infty$ if the constraint set is
	empty. Then
	\[
	\mathcal S_{h,\tau}^{\varepsilon,r,\delta}
	=
	\left\{
	z\in X:
	\Gamma_{h,\tau}^{r,\delta}(z)\le\varepsilon
	\right\}.
	\]
	This provides a numerical interpretation of the construction in terms of
	constrained linear optimization.

	\begin{proposition}
		\label{prop:properties-near-minimizing-threshold}
		For every $h,\tau,r>0$, $\varepsilon\ge0$, and $0<\delta\le1$, the following
		properties hold:
		\begin{itemize}
			\item[(i)]
			$\mathcal S_{h,\tau}^{\varepsilon,r,\delta}$ is compact;

			\item[(ii)]
			if $\varepsilon_1\le\varepsilon_2$, then
			$\mathcal S_{h,\tau}^{\varepsilon_1,r,\delta}
			\subset
			\mathcal S_{h,\tau}^{\varepsilon_2,r,\delta}$;

			\item[(iii)]
			if $\delta_1\le\delta_2$, then $\mathcal S_{h,\tau}^{\varepsilon,r,\delta_2}
			\subset
			\mathcal S_{h,\tau}^{\varepsilon,r,\delta_1};$

			\item[(iv)]
			if $r_1\le r_2$, then $	\mathcal S_{h,\tau}^{\varepsilon,r_1,\delta}
			\subset
			\mathcal S_{h,\tau}^{\varepsilon,r_2,\delta}.$
		\end{itemize}
	\end{proposition}

	\begin{proof}
		The action bound defining $\mathcal H_{h,\tau}^{\varepsilon}$ and the
		superlinearity of $L$ imply that
		$\widetilde{\mathcal H}_{h,\tau}^{\varepsilon}$ is tight. It is also narrowly
		closed, since the holonomy condition passes to the limit and the action is
		lower semicontinuous. Hence it is narrowly compact.

		The map $(\nu,z)\mapsto m_{\nu,r}(z)$ is continuous. Therefore, the supremum
		defining $Q_{h,\tau,r}^{\varepsilon}$ is attained and
		$Q_{h,\tau,r}^{\varepsilon}$ is continuous. It follows that
		$\mathcal S_{h,\tau}^{\varepsilon,r,\delta}$ is closed.

		To prove boundedness in the velocity variable, let
		$z=(x,v)\in\mathcal S_{h,\tau}^{\varepsilon,r,\delta}$. There exists
		$\nu\in\widetilde{\mathcal H}_{h,\tau}^{\varepsilon}$ such that
		$m_{\nu,r}(z)\ge\delta$, and hence $\nu(B_r(z))\ge\delta$. Setting
		$L_{\min}:=\inf_X L$, we obtain
		\[
		\bar L(h,\tau)+\varepsilon
		\ge
		\delta
		\inf_{\substack{y\in\mathbb T^d\\|\xi|\ge|v|-r}}
		L(y,\xi)
		+
		(1-\delta)L_{\min}.
		\]
		By superlinearity, this gives a bound on $|v|$. Since $\mathbb T^d$ is
		compact, $\mathcal S_{h,\tau}^{\varepsilon,r,\delta}$ is compact.

		Property (ii) follows from
		$\mathcal H_{h,\tau}^{\varepsilon_1}
		\subset\mathcal H_{h,\tau}^{\varepsilon_2}$.
		Property (iii) follows directly from the definition. Finally, if
		$r_1\le r_2$, the monotonicity of $\kappa$ gives
		$m_{\nu,r_1}\le m_{\nu,r_2}$ and proves (iv).
	\end{proof}

	We first establish the outer consistency of the approximation. For a set \(A\subset X\) and \(r>0\), we denote by \(A_r:=\{z\in X:d_X(z,A)<r\}\) its open \(r\)-neighborhood, where \(d_X(z,A):=\inf_{w\in A}d_X(z,w)\).

	\begin{proposition}
		\label{prop:near-minimizing-outer-consistency}
		Assume \eqref{eq:conv_parameter} and let $\varepsilon_n\to0$. For every fixed
		$r>0$ and $\delta>0$,
		\[
		\limsup_{n\to\infty}
		\mathcal S_{h_n,\tau_n}^{\varepsilon_n,r,\delta}
		\subset
		\bigl(\widetilde{\mathcal M}_L\bigr)_r.
		\]
	\end{proposition}

	\begin{proof}
		Let
		$z_n\in\mathcal S_{h_n,\tau_n}^{\varepsilon_n,r,\delta}$ and assume that
		$z_n\to z$. Choose
		$\mu_n\in\mathcal H_{h_n,\tau_n}^{\varepsilon_n}$ such that, setting
		$\widetilde\mu_n:=\mathcal R_{h_n,\tau_n}\mu_n$,
		$m_{\widetilde\mu_n,r}(z_n)\ge\delta-\frac1n.$

		The action bound
		$\int_XL\,d\widetilde\mu_n \le \bar L(h_n,\tau_n)+\varepsilon_n$
		and the superlinearity of $L$ imply tightness. Up to a subsequence,
		$\widetilde\mu_n\rightharpoonup\mu$. As in Proposition
		\ref{prop:fd-mather-measures-to-continuous}, the limit $\mu$ is closed.
		Moreover,
		\[
		\int_XL\,d\mu
		\le
		\liminf_{n\to\infty}
		\int_XL\,d\widetilde\mu_n
		\le
		-\alpha(H),
		\]
		so $\mu$ is a continuous Mather measure.

		Since \(z_n\to z\) and \(\kappa\) is Lipschitz, the functions
		\(w\mapsto\kappa(d_X(z_n,w)/r)\) converge uniformly to
		\(w\mapsto\kappa(d_X(z,w)/r)\). Hence
		\(m_{\mu,r}(z)\ge\delta>0\). Since \(\kappa\) vanishes on
		\([1,+\infty)\), there exists \(w\in\supp(\mu)\) such
		that \(d_X(z,w)<r\). As
		\(\supp(\mu)\subset\widetilde{\mathcal M}_L\), it follows
		that \(d_X(z,\widetilde{\mathcal M}_L)<r\), and therefore
		\(z\in(\widetilde{\mathcal M}_L)_r\).
	\end{proof}

	For the complementary inclusion, we construct fully--discrete holonomic measures converging to a prescribed Mather measure with vanishing action gap.

	\begin{lemma}
		\label{lem:fd-recovery-continuous-Mather}
		Assume \eqref{eq:conv_parameter}, and let $\mu$ be a continuous Mather
		measure. Then there exist $\nu_n\in\mathcal H_{h_n,\tau_n}$ such that
		\[
		\mathcal R_{h_n,\tau_n}\nu_n\rightharpoonup\mu,
		\qquad
		\mathcal A_{h_n,\tau_n}(\nu_n)\longrightarrow-\alpha(H).
		\]
		Consequently, the action gaps
		\[
		\eta_n
		:=
		\mathcal A_{h_n,\tau_n}(\nu_n)-\bar L(h_n,\tau_n)
		\]
		satisfy $\eta_n\to0$. More precisely, there exists $C>0$, depending only
		on $L$ and on $\supp(\mu)$, such that 
		\begin{equation}\label{eq:recovery-gap-rate}
			0\le\eta_n\le C\left(\tau_n+\frac{h_n}{\tau_n}\right).
		\end{equation}
	\end{lemma}

\begin{proof}
	Let \(\Phi_L^t\) be the Euler--Lagrange flow. The Mather measure
	\(\mu\) is invariant under \(\Phi_L^t\) and has compact support; see
	\cite{Fathi2008,Sorrentino2015}. For
	\((x,v)\in\supp(\mu)\), write
	\[
	\Phi_L^{\tau_n}(x,v)=(y_n,w_n)
	\]
	and let \(\gamma_{x,v}:[0,\tau_n]\to\mathbb T^d\) be the corresponding
	Euler--Lagrange trajectory. Choose a lift
	\(\widetilde\gamma_{x,v}:[0,\tau_n]\to\mathbb R^d\) and set
	\[
	\widetilde x:=\widetilde\gamma_{x,v}(0),
	\qquad
	\widetilde y_n:=\widetilde\gamma_{x,v}(\tau_n),
	\qquad
	\Delta_n(x,v):=\widetilde y_n-\widetilde x.
	\]
	
	Let \(q_n:\mathbb T^d\to\mathcal G_{h_n}\) be a measurable
	nearest-grid-point projection, with ties resolved lexicographically.
	Choose lifts \(\widetilde q_n(x)\) and
	\(\widetilde q_n(y_n)\) satisfying
	\[
	|\widetilde q_n(x)-\widetilde x|
	+
	|\widetilde q_n(y_n)-\widetilde y_n|
	\le Ch_n.
	\]
	Write
	\[
	\widetilde q_n(x)=q_n(x)+m_n^-(x,v),
	\qquad
	\widetilde q_n(y_n)=q_n(y_n)+m_n^+(x,v),
	\]
	where \(m_n^\pm(x,v)\in\mathbb Z^d\), and define
	\[
	\ell_n(x,v):=m_n^+(x,v)-m_n^-(x,v).
	\]
	Then
	\[
	q_n(y_n)+\ell_n(x,v)-q_n(x)
	=
	\widetilde q_n(y_n)-\widetilde q_n(x)
	=
	\Delta_n(x,v)+e_n(x,v),
	\]
	where \(|e_n(x,v)|\le Ch_n\), uniformly on \(\supp(\mu)\).
	The construction is illustrated in
	Figure~\ref{fig:recovery-winding-construction}.

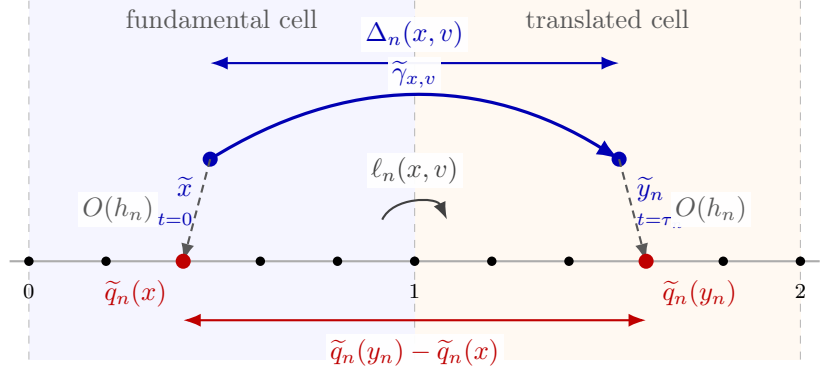
\begin{figure}[ht]
	\centering
	\begin{tikzpicture}[
		x=1.02cm,
		y=1cm,
		>=Latex,
		gridpoint/.style={
			circle,
			fill=black,
			inner sep=1.25pt
		},
		continuouspoint/.style={
			circle,
			fill=blue!70!black,
			inner sep=2pt
		},
		projectedpoint/.style={
			circle,
			fill=red!75!black,
			inner sep=2pt
		},
		trajectory/.style={
			blue!70!black,
			very thick,
			->
		},
		projection/.style={
			gray!70!black,
			densely dashed,
			thick,
			->
		},
		continuous displacement/.style={
			blue!70!black,
			thick,
			<->
		},
		discrete displacement/.style={
			red!75!black,
			thick,
			<->
		},
		cell boundary/.style={
			gray!55,
			dashed
		},
		text box/.style={
			fill=white,
			fill opacity=0.96,
			text opacity=1,
			inner sep=1.5pt
		},
		every node/.style={
			font=\small
		}
		]
		
		\fill[blue!4]
		(0,-1.30) rectangle (5,3.45);
		\fill[orange!5]
		(5,-1.30) rectangle (10,3.45);
		
		\draw[cell boundary]
		(0,-1.30) -- (0,3.45);
		\draw[cell boundary]
		(5,-1.30) -- (5,3.45);
		\draw[cell boundary]
		(10,-1.30) -- (10,3.45);
		
		\node[
		gray!70!black,
		font=\small
		] at (2.5,3.20)
		{fundamental cell};
		
		\node[
		gray!70!black,
		font=\small
		] at (7.5,3.20)
		{translated cell};
		
		\coordinate (xcont) at (2.35,1.35);
		\coordinate (ycont) at (7.65,1.35);
		\coordinate (xgrid) at (2,0);
		\coordinate (ygrid) at (8,0);
		
		\draw[continuous displacement]
		(2.35,2.62)
		--
		(7.65,2.62)
		node[
		midway,
		above=4pt,
		text box
		]
		{\(\Delta_n(x,v)\)};
		
		\draw[trajectory]
		(xcont)
		to[out=32,in=148]
		node[
		pos=0.50,
		above=1pt,
		blue!70!black,
		text box
		]
		{\(\widetilde\gamma_{x,v}\)}
		(ycont);
		
		\node[continuouspoint] at (xcont) {};
		\node[continuouspoint] at (ycont) {};
		
		\node[
		blue!70!black,
		anchor=north east,
		align=right
		] at ($(xcont)+(-0.10,-0.10)$)
		{\(\widetilde x\)\\[-1pt]
			\(\scriptstyle t=0\)};
		
		\node[
		blue!70!black,
		anchor=north west,
		align=left
		] at ($(ycont)+(0.10,-0.10)$)
		{\(\widetilde y_n\)\\[-1pt]
			\(\scriptstyle t=\tau_n\)};
		
		\draw[gray!65,thick]
		(-0.25,0) -- (10.25,0);
		
		\foreach \xi in {0,1,...,10}{
			\node[gridpoint] at (\xi,0) {};
		}
		
		\node[
		below=5pt,
		font=\scriptsize
		] at (0,0) {\(0\)};
		
		\node[
		below=5pt,
		font=\scriptsize
		] at (5,0) {\(1\)};
		
		\node[
		below=5pt,
		font=\scriptsize
		] at (10,0) {\(2\)};
		
		\draw[projection]
		(xcont) -- (xgrid);
		
		\draw[projection]
		(ycont) -- (ygrid);
		
		\node[
		gray!70!black,
		anchor=east,
		text box
		] at (1.68,0.68)
		{\(O(h_n)\)};
		
		\node[
		gray!70!black,
		anchor=west,
		text box
		] at (8.32,0.68)
		{\(O(h_n)\)};
		
		\node[projectedpoint] at (xgrid) {};
		\node[projectedpoint] at (ygrid) {};
		
		\node[
		red!75!black,
		anchor=north east
		] at ($(xgrid)+(-0.08,-0.12)$)
		{\(\widetilde q_n(x)\)};
		
		\node[
		red!75!black,
		anchor=north west
		] at ($(ygrid)+(0.08,-0.12)$)
		{\(\widetilde q_n(y_n)\)};
		
		\draw[
		black!75,
		thick,
		->
		]
		(4.58,0.55)
		to[
		out=65,
		in=115
		]
		node[
		midway,
		above=5pt,
		text box
		]
		{\(\ell_n(x,v)\)}
		(5.42,0.55);
		
		\draw[discrete displacement]
		(2,-0.78)
		--
		(8,-0.78)
		node[
		midway,
		below=5pt,
		text box
		]
		{\(\widetilde q_n(y_n)-\widetilde q_n(x)\)};
		
	\end{tikzpicture}
	
	\caption{
		Construction of the winding-labeled transition in the universal
		cover. The continuous and projected displacements differ by
		\(O(h_n)\), while \(\ell_n(x,v)\) records the change of
		fundamental cell.
	}
	\label{fig:recovery-winding-construction}
\end{figure}

	
	Define \(\nu_n\) as the pushforward of \(\mu\) by
	\[
	(x,v)
	\longmapsto
	\bigl(q_n(x),q_n(y_n),\ell_n(x,v)\bigr).
	\]
	The maps involved are measurable by the choice of the projection and
	the lexicographic tie-breaking rule.
	
	For every \(\varphi:\mathcal G_{h_n}\to\mathbb R\), the invariance of
	\(\mu\) under \(\Phi_L^{\tau_n}\) gives
	\[
	\int_X
	\bigl(\varphi(q_n(y_n))-\varphi(q_n(x))\bigr)
	\,d\mu(x,v)
	=0.
	\]
	Hence \(\nu_n\in\mathcal H_{h_n,\tau_n}\). The velocity associated with the discrete transition is
	\[
	v_n(x,v)
	:=
	\frac{q_n(y_n)+\ell_n(x,v)-q_n(x)}{\tau_n}
	=
	\frac{\widetilde q_n(y_n)-\widetilde q_n(x)}{\tau_n}.
	\]
	By construction,
	\[
	v_n(x,v)
	=
	\frac{\Delta_n(x,v)}{\tau_n}
	+
	O\left(\frac{h_n}{\tau_n}\right).
	\]
	The smoothness of the Euler--Lagrange flow and the compactness of
	\(\supp(\mu)\) give
	\[
	\frac{\Delta_n(x,v)}{\tau_n}
	=
	v+O(\tau_n)
	\]
	uniformly on \(\supp(\mu)\). Therefore
	\[
	v_n(x,v)
	=
	v+O(\tau_n)
	+
	O\left(\frac{h_n}{\tau_n}\right)
	\]
	uniformly on \(\supp(\mu)\). It follows that \(q_n(x)\to x\) and \(v_n(x,v)\to v\) uniformly on
	\(\supp(\mu)\). Consequently,
	\[
	\mathcal R_{h_n,\tau_n}\nu_n
	\rightharpoonup\mu
	\]
	narrowly, and
	\[
	\mathcal A_{h_n,\tau_n}(\nu_n)
	=
	\int_X L(q_n(x),v_n(x,v))\,d\mu(x,v)
	\longrightarrow
	\int_XL\,d\mu
	=
	-\alpha(H).
	\]
	Finally, Proposition \ref{prop:fd-critical-duality} and Theorem
	\ref{thm:continuous-fully-discrete-critical-error} imply
	\(0\le\eta_n\to0\).	The construction also gives a quantitative estimate. Indeed, since
	\[
	d_{\mathbb T^d}(q_n(x),x)\le Ch_n
	\]
	and
	\[
	|v_n(x,v)-v|
	\le
	C\left(\tau_n+\frac{h_n}{\tau_n}\right)
	\]
	uniformly on \(\supp(\mu)\), the Lipschitz continuity of \(L\) on a
	suitable compact velocity set yields
	\[
	\begin{aligned}
		\left|
		\mathcal A_{h_n,\tau_n}(\nu_n)+\alpha(H)
		\right|
		&=
		\left|
		\int_X
		\bigl(
		L(q_n(x),v_n(x,v))-L(x,v)
		\bigr)
		\,d\mu
		\right| \\
		&\le
		C\left(
		\tau_n+\frac{h_n}{\tau_n}
		\right).
	\end{aligned}
	\]
	Combining this estimate with Theorem
	\ref{thm:continuous-fully-discrete-critical-error} gives
	\eqref{eq:recovery-gap-rate}.
\end{proof}
We now pass from the recovery of individual measures to that of the whole Mather set. Set 
\[ 
Y:=\bigcup_{\mu\ \mathrm{Mather}}\supp(\mu), \qquad \widetilde{\mathcal M}_L=\overline Y.
 \]
 Since \(X\) is a separable metric space, it has a countable basis. For each basis element intersecting \(Y\), choose a Mather measure whose support intersects it. This gives a sequence \((\mu_k)_{k\ge1}\) such that 
 \[ 
 \widetilde{\mathcal M}_L = \overline{\bigcup_{k\ge1}\supp(\mu_k)}. 
 \]
 Let \(a_k>0\) with \(\sum_{k\ge1}a_k=1\), and set 
 \[ \mu_*:=\sum_{k\ge1}a_k\mu_k. \] 
 By linearity of the closedness condition and of the action, \(\mu_*\) is a Mather measure. Moreover, since every \(a_k\) is positive,
  \[ \supp(\mu_*) = \overline{\bigcup_{k\ge1}\supp(\mu_k)} = \widetilde{\mathcal M}_L. \] 
  We call \(\mu_*\) a full-support Mather measure.

	\begin{proposition}
		\label{prop:near-minimizing-recovery}
		Assume \eqref{eq:conv_parameter}, and let $\mu_*$ be a full-support Mather
		measure. Let $(\nu_n)_n$ be the recovery sequence given by Lemma
		\ref{lem:fd-recovery-continuous-Mather}, and set
		\[
		\eta_n
		:=
		\mathcal A_{h_n,\tau_n}(\nu_n)-\bar L(h_n,\tau_n).
		\]
		If $\varepsilon_n\to0$ and $\eta_n\le\varepsilon_n$ for all sufficiently
		large $n$, then, for every fixed $r>0$, there exists $\delta_r>0$ such that
		\[
		\widetilde{\mathcal M}_L
		\subset
		\mathcal S_{h_n,\tau_n}^{\varepsilon_n,r,\delta_r}
		\]
		for all sufficiently large $n$.
	\end{proposition}

	\begin{proof}
		Let $\widetilde\nu_n:=\mathcal R_{h_n,\tau_n}\nu_n$. Then
		$\widetilde\nu_n\rightharpoonup\mu_*$ and $\eta_n\to0$. Fix $r>0$ and choose
		$z_1,\ldots,z_N\in\widetilde{\mathcal M}_L$ such that
		\[
		\widetilde{\mathcal M}_L
		\subset
		\bigcup_{k=1}^N B_{r/4}(z_k).
		\]
		Since $\supp(\mu_*)=\widetilde{\mathcal M}_L$, one may choose
		$0<\delta_r< \min_{1\le k\le N} \mu_*(B_{r/4}(z_k)).$
		By narrow convergence, for all sufficiently large $n$,
		$\widetilde\nu_n(B_{r/4}(z_k))>\delta_r, \qquad k=1,\ldots,N.$

		Given $z\in\widetilde{\mathcal M}_L$, choose $k$ such that
		$z\in B_{r/4}(z_k)$. Since
		$B_{r/4}(z_k)\subset B_{r/2}(z)$,
		$m_{\widetilde\nu_n,r}(z) \ge \widetilde\nu_n(B_{r/2}(z)) > \delta_r.$
		Moreover, $\eta_n\le\varepsilon_n$ implies
		$\nu_n\in\mathcal H_{h_n,\tau_n}^{\varepsilon_n}$. Hence
		$Q_{h_n,\tau_n,r}^{\varepsilon_n}(z)>\delta_r,$
		which proves the claim.
	\end{proof}

	\begin{remark}
		\label{rem:choice-near-minimizing-parameters}
		By \eqref{eq:recovery-gap-rate}, one may choose
		$\varepsilon_n=C_0(\tau_n+h_n/\tau_n)$, with $C_0$ sufficiently large.
		This tolerance dominates the recovery gap and vanishes under
		\eqref{eq:conv_parameter}; in the balanced regime $h_n\asymp\tau_n^2$,
		it is $O(\tau_n)$. The radius $r$ fixes the phase-space resolution, while
		$\delta_r$ is a uniform lower bound on the local mass of the chosen
		full-support Mather measure at that scale.
	\end{remark}

	Combining recovery and outer consistency gives:

	\begin{corollary}
		\label{cor:near-minimizing-fixed-scale-approximation}
		Under the assumptions of Proposition
		\ref{prop:near-minimizing-recovery}, for every fixed $r>0$,
		\[
		\widetilde{\mathcal M}_L
		\subset
		\liminf_{n\to\infty}
		\mathcal S_{h_n,\tau_n}^{\varepsilon_n,r,\delta_r}
		\subset
		\limsup_{n\to\infty}
		\mathcal S_{h_n,\tau_n}^{\varepsilon_n,r,\delta_r}
		\subset
		\bigl(\widetilde{\mathcal M}_L\bigr)_r.
		\]
	\end{corollary}

	\begin{proof}
		The first inclusion follows from Proposition
		\ref{prop:near-minimizing-recovery}, and the second is a general property of
		Kuratowski limits. Since $\varepsilon_n\to0$, the final inclusion follows
		from Proposition \ref{prop:near-minimizing-outer-consistency}.
	\end{proof}
%

	\begin{remark}
		\label{rem:obstruction-vanishing-radius}
		Corollary \ref{cor:near-minimizing-fixed-scale-approximation} is a
		fixed-scale statement, and a diagonal choice $r_n\to0$ is not available
		within the present construction. Indeed, the recovery step forces
		$\delta\le\delta_{r}$, and $\delta_r\to0$ as $r\to0$ whenever
		$\widetilde{\mathcal M}_L$ is not purely atomic, since $\delta_r$ is
		bounded above by the local mass of $\mu_*$ at scale $r$. On the other
		hand, the proof of Proposition
		\ref{prop:near-minimizing-outer-consistency} uses $\delta>0$ in the
		limit, and its conclusion degenerates when $\delta_n\to0$.

		The localization estimate \eqref{eq:implementation-localization} below
		suggests how the two requirements may be reconciled: if $r_n\to0$,
		$\delta_n\to0$ and $\varepsilon_n/\delta_n\to0$, then every point of
		$\limsup_n\mathcal S_{h_n,\tau_n}^{\varepsilon_n,r_n,\delta_n}$ is a
		limit of transitions whose Bellman defect tends to zero, and the argument
		of Proposition \ref{prop:fd-mather-limsup-mane} may be adapted to
		almost-calibrated chains, leading naturally to the Ma\~n\'e set. Combining this with
		\eqref{eq:recovery-gap-rate}, such a choice is possible as soon as
		$\eta_n\ll\delta_{r_n}$, for instance under a lower bound of the form
		$\mu_*(B_r(z))\ge c\,r^{s}$ on $\supp(\mu_*)$. We do not pursue this
		here.
	\end{remark}

	\begin{remark}
		\label{rem:near-minimizing-versus-fd-Mather}
		The sets $\mathcal S_{h,\tau}^{\varepsilon,r,\delta}$ are not
		fully--discrete Mather sets. The fully--discrete Mather set remains the
		natural object for the intrinsic discrete theory. The present construction
		serves a different purpose: the tolerance $\varepsilon$ prevents the grid
		from discarding asymptotically minimizing components, while the threshold
		$\delta$ excludes regions carrying negligible local mass. It therefore
		provides a more stable approximation of the full continuous Mather set.
	\end{remark}

\section{Implementation of the Mass-Threshold Approximation}
\label{sec:implementation}

We now turn the construction of
Section~\ref{sec:near-minimizing-mass-threshold} into a family of
finite-dimensional linear programs.

Writing \(x_i=ih\), with \(Nh=1\), every lifted displacement has the
form \(kh\), \(k\in\mathbb Z^d\). It determines
\[
j=(i+k)\bmod N,
\qquad
\ell=\frac{k-(j-i)}{N},
\qquad
v_{i,k}=\frac{kh}{\tau}.
\]
Thus transitions may be indexed by \((i,k)\) and represented in phase
space by
\[
z_{i,k}:=\left(x_i,\frac{kh}{\tau}\right).
\]
For a velocity cutoff \(V>0\), set
\[
\mathcal E_V
:=
\left\{(i,k):\left|\frac{kh}{\tau}\right|\le V\right\}.
\]
The resulting transition graph is finite. The coercivity estimate in
Proposition~\ref{prop:properties-near-minimizing-threshold} bounds the
velocity component of the test region, while
\eqref{eq:bounded-jump} shows that a sufficiently large cutoff
preserves the minimum action. In practice, \(V\) is increased until
the critical value and the detected set stabilize.

On \(\mathcal E_V\), we impose the positivity, normalization, and
balance conditions in \eqref{eq:fd-holonomy-winding}. The objective is
the action \eqref{eq:fd-Mather-LP-winding}, the near-minimality
constraint is given in
Definition~\ref{def:fd-near-minimizing-measures}, and the regularized
local mass is defined by
\eqref{eq:near-min-regularized-local-mass}. Hence
\(Q_{h,\tau,r}^{\varepsilon}(z)\) is the value of a finite-dimensional
linear program.

\begin{remark}
	\label{rem:implementation-Bellman-defect}
	Let \((u_{h,\tau},\bar L(h,\tau))\) solve
	\eqref{eq:fully-discrete-ergodic-equation} and define the reduced cost
	or Bellman defect by
	\[
	\overline c_{ij}^{\ell}
	:=
	L(x_i,v_{ij}^{\ell})-\bar L(h,\tau)
	+
	\frac{u_{h,\tau}(x_i)-u_{h,\tau}(x_j)}{\tau}.
	\]
	The Bellman inequality gives
	\(\overline c_{ij}^{\ell}\ge0\), and holonomy yields
	\begin{equation}
		\label{eq:bellman-gap-identity}
		\mathcal A_{h,\tau}(\mu)-\bar L(h,\tau)
		=
		\sum_{i,j,\ell}
		\overline c_{ij}^{\ell}\mu_{ij}^{\ell}.
	\end{equation}
	Thus the action gap is the average Bellman defect.
	
	If
	\(z\in\mathcal S_{h,\tau}^{\varepsilon,r,\delta}\), an admissible
	measure assigns at least mass \(\delta\) to transitions in \(B_r(z)\)
	and has average defect at most \(\varepsilon\). At least one such
	transition must therefore satisfy
	\(\overline c_{ij}^{\ell}\le\varepsilon/\delta\). Consequently,
	\begin{equation}
		\label{eq:implementation-localization}
		\mathcal S_{h,\tau}^{\varepsilon,r,\delta}
		\subset
		\bigcup_{\overline c_{ij}^{\ell}\le\varepsilon/\delta}
		B_r\bigl((x_i,v_{ij}^{\ell})\bigr).
	\end{equation}
	This necessary condition provides a preliminary screening of the test
	region.
\end{remark}

\begin{lemma}
	\label{lem:implementation-Q-lipschitz}
	For every \(z,z'\in X\),
	\[
	\left|
	Q_{h,\tau,r}^{\varepsilon}(z)
	-
	Q_{h,\tau,r}^{\varepsilon}(z')
	\right|
	\le
	\frac{\operatorname{Lip}(\kappa)}{r}\,d_X(z,z').
	\]
\end{lemma}

\begin{proof}
	For every
	\(\widetilde\mu\in\widetilde{\mathcal H}_{h,\tau}^{\varepsilon}\),
	\[
	\begin{aligned}
		\left|
		m_{\widetilde\mu,r}(z)-m_{\widetilde\mu,r}(z')
		\right|
		&\le
		\int_X
		\left|
		\kappa\left(\frac{d_X(z,w)}{r}\right)
		-
		\kappa\left(\frac{d_X(z',w)}{r}\right)
		\right|
		\,d\widetilde\mu(w) \\
		&\le
		\frac{\operatorname{Lip}(\kappa)}{r}\,d_X(z,z').
	\end{aligned}
	\]
	Taking the supremum over
	\(\widetilde{\mathcal H}_{h,\tau}^{\varepsilon}\) proves the claim.
\end{proof}

The implementation consists of building the truncated transition
graph, computing the critical value and the Bellman defects, screening
the test region by \eqref{eq:implementation-localization}, and solving
the local linear program only at the remaining test points. The
Lipschitz estimate permits the classification obtained at one point
to be propagated to a neighborhood whenever the computed value is
strictly separated from \(\delta\).

The convergence regime requires
\(\tau\to0\) and \(h/\tau\to0\), while the upper-support result of
Section~\ref{sec:convergence-fd-mather} uses the stronger scaling
\eqref{eq:conv_parameter_strong}. Since the recovery construction
gives
\[
v_{h,\tau}
=
v+O(\tau)+O\left(\frac{h}{\tau}\right),
\]
balancing the two errors suggests $h\asymp\tau^2$. Theorem~\ref{thm:continuous-fully-discrete-critical-error} and
\eqref{eq:recovery-gap-rate} then suggest the action tolerance
\[
\varepsilon_{h,\tau}
=
C_0\left(\tau+\frac{h}{\tau}\right),
\]
which is \(O(\tau)\) in the balanced regime.

The radius \(r\) determines the phase-space resolution and should not
be substantially smaller than the velocity spacing \(h/\tau\); a
natural choice is
\[
r\ge c\,\frac{h}{\tau},
\qquad c>1.
\]
The mass threshold \(\delta\) controls selectivity, while the ratio
\(\varepsilon/\delta\) determines the screening level in
\eqref{eq:implementation-localization}. Finally, \(V\) is increased
until the critical value and the computed approximation remain
unchanged.

\subsection{A model example with two minimizing components}
\label{subsec:implementation-example}

\textit{The example only illustrates the mass-threshold mechanism. A systematic study
	of the algorithm, its relation to computational KAM methods, and applications
	to invariant tori is left for future work.}

Let \(d=1\) and consider
\[
L(x,v)=\frac12|v|^2-W(x),
\]
where \(W\) is a smooth periodic potential consisting of two
identical, disjoint bumps centered at
\[
x_1=\frac14,
\qquad
x_2=\frac{29}{48},
\]
and normalized so that
$W(x_1)=W(x_2)=\max_{\mathbb T}W=1$. 
Thus
\[
\widetilde{\mathcal M}_L
=
\{(x_1,0),(x_2,0)\},
\qquad
-\alpha(H)=-1.
\]
The numerical parameters are summarized in
Table~\ref{tab:numerical-parameters}. Notice that
\(x_1=30/120\) is a grid point, whereas
\(x_2=72.5/120\) lies halfway between the neighboring nodes
\(72/120\) and \(73/120\).

\begin{table}[ht]
	\centering
	\begin{tabular}{ccl}
		\hline
		\textbf{Parameter} & \textbf{Value} & \textbf{Description} \\
		\hline
		\(d\) & \(1\) & space dimension \\
		\(N\) & \(120\) & number of spatial grid points \\
		\(h\) & \(1/120\) & spatial mesh size \\
		\(\tau\) & \(5\cdot10^{-2}\) & time step \\
		\(V\) & \(1.2\) & velocity cutoff \\
		\(x_1\) & \(1/4\) & first maximum, on the grid \\
		\(x_2\) & \(29/48\) & second maximum, at half a cell \\
		\(\mathcal K_V\) & \(\{-7,\ldots,7\}\) &
		admissible lifted displacements \\
		\(\#\mathcal K_V\) & \(15\) &
		displacements per node \\
		\(M\) & \(1800\) &
		arcs in the truncated graph \\
		\hline
	\end{tabular}
	\caption{Parameters of the numerical example.}
	\label{tab:numerical-parameters}
\end{table}
Moreover, \(h/\tau=1/6\), and hence
\[
\mathcal K_V
=
\left\{
k\in\mathbb Z:
\left|\frac{kh}{\tau}\right|\le V
\right\}
=
\{-7,\ldots,7\}.
\]
Thus each node has \(15\) outgoing arcs and $M=N\#\mathcal K_V=1800.$

Let \(\overline c_1=0\) be the reduced cost of the stationary
self-loop at \(x_1\), and let \(\overline c_2>0\) be the minimum
reduced cost among the transitions in \(B_r((x_2,0))\). Since
\(\kappa\le\mathbf 1_{[0,1)}\), identity
\eqref{eq:bellman-gap-identity} gives
\[
m_{\mathcal R_{h,\tau}\mu,r}(x_2,0)
\le
\frac{\mathcal A_{h,\tau}(\mu)-\bar L(h,\tau)}
{\overline c_2}.
\]
Therefore
\[
Q_{h,\tau,r}^{\varepsilon}(x_2,0)
\le
\min\left\{1,\frac{\varepsilon}{\overline c_2}\right\}.
\]
In this example equality is attained by combining the zero-defect
measure at the first component with a stationary measure near the
second component. 

We use the piecewise affine kernel
\[
\kappa(s)=\min\{1,\max\{0,2-2s\}\},
\qquad
\operatorname{Lip}(\kappa)=2.
\]
The critical value is computed by solving the linear-programming
formulation of the minimum-cycle-mean problem on the transition graph;
see \cite{karp}. The dual variables provide a critical potential and
the corresponding reduced costs. Relative value iteration gives an
independent check of the critical value. The function
\(Q_{h,\tau,r}^{\varepsilon}\) is evaluated on a phase-space test set
of \(5025\) points using the truncated graph \(\mathcal E_V\).

\begin{figure}[ht]
	\centering
	\includegraphics[width=\textwidth]{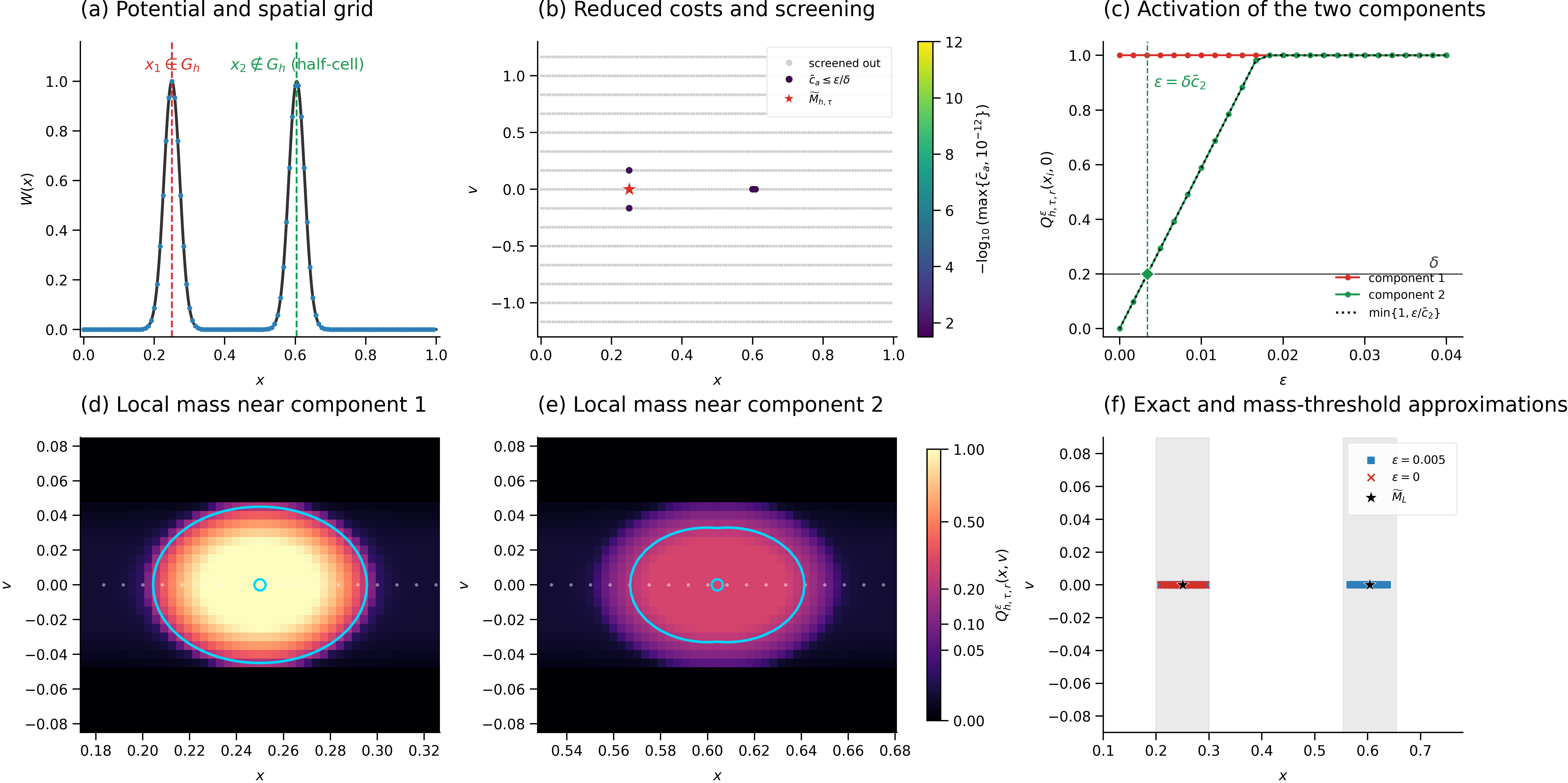}
	\caption{
		(a) Potential and grid; (b) reduced costs and screening;
		(c) activation of the components; (d), (e) local mass fields;
		(f) exact and near-minimizing approximations.
	}
	\label{fig:numerics}
\end{figure}

For this discretization,
\[
\bar L(h,\tau)=-1,
\qquad
\overline c_1=0,
\qquad
\overline c_2=1.6988\cdot10^{-2}.
\]
As shown in Figure~\ref{fig:numerics}(a), the grid represents the
first maximum exactly but not the second. Consequently, the exact
fully--discrete Mather set consists only of \((x_1,0)\).

Figure~\ref{fig:numerics}(b) shows the atoms of the truncated graph.
Only five arcs satisfy
\(\overline c_a\le\varepsilon/\delta\): three at \(x_1\), with
velocities \(-h/\tau\), \(0\), and \(h/\tau\), and the two stationary
self-loops at the nodes adjacent to \(x_2\). The screening condition
localizes the region where \(Q_{h,\tau,r}^{\varepsilon}\) may exceed
\(\delta\), thereby reducing the number of test points at which the
local linear program must be solved.

Figure~\ref{fig:numerics}(c) shows that
\[
Q_{h,\tau,r}^{\varepsilon}(x_1,0)=1,
\qquad
Q_{h,\tau,r}^{\varepsilon}(x_2,0)
=
\min\left\{1,\frac{\varepsilon}{\overline c_2}\right\}.
\]
Thus the second component crosses the threshold when
\(\varepsilon=\delta\overline c_2\).

For $r=5\cdot10^{-2}$, 
$\delta=0.2$, 
$\varepsilon=5\cdot10^{-3}$,
we obtain
\[
Q_{h,\tau,r}^{\varepsilon}(x_1,0)=1,
\qquad
Q_{h,\tau,r}^{\varepsilon}(x_2,0)=0.2943>\delta.
\]
The corresponding fields and level sets
\(Q_{h,\tau,r}^{\varepsilon}=\delta\) are shown in
Figures~\ref{fig:numerics}(d)--(e). Finally,
Figure~\ref{fig:numerics}(f) shows that exact discrete minimization
detects only the first component, whereas the mass-threshold
approximation recovers both continuous minimizing components.
	\appendix
	\section{A technical lemma}
	\label{app:technical-lemma}

	Throughout this appendix $\pi:\mathbb R^d\to\mathbb T^d$ denotes the canonical
	projection. We say that a bi-infinite sequence of transitions
	$(x_k^n,x_{k+1}^n,\ell_k^n)_{k\in\mathbb Z}$, with
	$x_k^n\in\mathcal G_{h_n}$ and $\ell_k^n\in\mathbb Z^d$, is a
	\emph{fully--discrete calibrated chain} for
	$(u_{h_n,\tau_n},\bar L(h_n,\tau_n))$ if every transition realizes the minimum
	in \eqref{eq:fully-discrete-ergodic-equation}, that is
	\begin{equation}\label{eq:chain-calibration}
		u_{h_n,\tau_n}(x_{k+1}^n)-u_{h_n,\tau_n}(x_k^n)
		=
		\tau_nL\bigl(x_k^n,v_k^n\bigr)-\tau_n\bar L(h_n,\tau_n),
		\qquad
		v_k^n:=\frac{x_{k+1}^n+\ell_k^n-x_k^n}{\tau_n},
	\end{equation}
	for every $k\in\mathbb Z$.

\begin{lemma}
	\label{lem:compactness-pointed-calibrated-chains}
	Assume \eqref{eq:conv_parameter_strong} and let
	\(\widetilde u_{h_n,\tau_n}\to u\) uniformly on \(\mathbb T^d\),
	where \(u\) is a critical solution. Let
	\((x_k^n,x_{k+1}^n,\ell_k^n)_{k\in\mathbb Z}\) be fully--discrete
	calibrated chains and set
	\[
	v_k^n
	:=
	\frac{x_{k+1}^n+\ell_k^n-x_k^n}{\tau_n}.
	\]
	If $(x_0^n,v_0^n)\longrightarrow(x,v)$, 
	then, up to a subsequence, their piecewise affine interpolations
	\(\gamma_n\) converge locally uniformly to a globally
	\(u\)-calibrated Euler--Lagrange trajectory
	\(\gamma:\mathbb R\to\mathbb T^d\). Moreover, setting
	\[
	V_n(t):=v_k^n,
	\qquad
	t\in[k\tau_n,(k+1)\tau_n),
	\]
	one has
	$V_n\longrightarrow\dot\gamma$ in  $ L^\infty_{\mathrm{loc}}(\mathbb R)$,
	and $\gamma(0)=x$, $\dot\gamma(0)=v.$
	Consequently, $(x,v)\in\widetilde{\mathcal I}(u)
	\subset\widetilde{\mathcal N}_L$.
\end{lemma}

\begin{proof}
	Choose a lift $\widetilde x\in\mathbb R^d$ of $x$ and lifts
	$\widetilde x_0^n\to\widetilde x$ of $x_0^n$. Define inductively
	\(
	\widetilde x_{k+1}^n-\widetilde x_k^n=\tau_n v_k^n
	\)
	for all $k\in\mathbb Z$, so that
	$\pi(\widetilde x_k^n)=x_k^n$. Let $\widetilde\gamma_n$ be the piecewise
	affine interpolation of the points $\widetilde x_k^n$, set
	$\gamma_n=\pi\circ\widetilde\gamma_n$, and denote by $V_n$ the
	piecewise constant velocity. Then
	$\dot{\widetilde\gamma}_n=V_n$ a.e.
	
	By Proposition~\ref{prop:fd-lip-bound}, every calibrated transition
	satisfies
	\begin{equation}
		\label{eq:calibrated-chain-velocity-bound}
		|v_k^n|\le D,
		\qquad k\in\mathbb Z,
	\end{equation}
	with $D$ independent of $k$ and $n$.
	
	To estimate the variation of consecutive velocities, let
	$\widetilde L$ be the periodic lift of $L$ and define
	\[
	S_k^n(y):=
	\tau_n\widetilde L\!\left(
	\widetilde x_k^n,\frac{y-\widetilde x_k^n}{\tau_n}
	\right)
	+
	\tau_n\widetilde L\!\left(
	y,\frac{\widetilde x_{k+2}^n-y}{\tau_n}
	\right).
	\]
	The Bellman inequalities imply that
	$\widetilde x_{k+1}^n$ minimizes $S_k^n$ on the lifted grid, since
	equality is attained there by calibration. Using
	\eqref{eq:calibrated-chain-velocity-bound},
	\eqref{eq:conv_parameter_strong}, and the $C^2$ regularity of $L$, one
	obtains
	\[
	\sup_{|y-\widetilde x_{k+1}^n|\le h_n}
	\|\nabla^2 S_k^n(y)\|
	\le \frac{C}{\tau_n}.
	\]
	Comparing the minimum with neighboring grid points and applying Taylor's
	formula yields
	\begin{equation}
		\label{eq:approximate-discrete-EL}
		|\nabla S_k^n(\widetilde x_{k+1}^n)|
		\le C\frac{h_n}{\tau_n}.
	\end{equation}
	
	Since
	\[
	\nabla S_k^n(\widetilde x_{k+1}^n)
	=
	D_vL(x_k^n,v_k^n)
	-
	D_vL(x_{k+1}^n,v_{k+1}^n)
	+
	\tau_nD_xL(x_{k+1}^n,v_{k+1}^n),
	\]
	the bound \eqref{eq:approximate-discrete-EL}, together with
	$d_{\mathbb T^d}(x_k^n,x_{k+1}^n)\le D\tau_n$, gives
	\[
	|D_vL(x_{k+1}^n,v_k^n)-D_vL(x_{k+1}^n,v_{k+1}^n)|
	\le
	C\!\left(\tau_n+\frac{h_n}{\tau_n}\right).
	\]
	Uniform strict convexity of $L$ on
	$\mathbb T^d\times\overline B_D$ therefore implies
	\[
	|v_{k+1}^n-v_k^n|
	\le
	C\!\left(\tau_n+\frac{h_n}{\tau_n}\right).
	\]
	Using $h_n\le C_*\tau_n^2$, we obtain
	\begin{equation}
		\label{eq:discrete-velocity-increment}
		|v_{k+1}^n-v_k^n|\le C\tau_n .
	\end{equation}
	
	Let $\widehat V_n$ be the piecewise affine interpolation of the values
	$v_k^n$ at the times $k\tau_n$. By
	\eqref{eq:calibrated-chain-velocity-bound} and
	\eqref{eq:discrete-velocity-increment},
	$\|\widehat V_n\|_\infty\le D$ and
	$\operatorname{Lip}(\widehat V_n)\le C$, while
	$\|V_n-\widehat V_n\|_\infty\le C\tau_n$. Hence, after extraction,
	$\widehat V_n\to V$ locally uniformly on $\mathbb R$ for some Lipschitz
	function $V$, and therefore
	\begin{equation}
		\label{eq:piecewise-velocity-convergence}
		V_n\to V
		\qquad\text{in }L^\infty_{\mathrm{loc}}(\mathbb R).
	\end{equation}
	
	Since
	$\widetilde\gamma_n(t)=\widetilde x_0^n+\int_0^tV_n(s)\,ds$,
	the curves $\widetilde\gamma_n$ converge locally uniformly to
	\[
	\widetilde\gamma(t)
	=
	\widetilde x+\int_0^tV(s)\,ds.
	\]
	Setting $\gamma=\pi\circ\widetilde\gamma$, we obtain
	$\dot\gamma=V$, $\gamma(0)=x$, and, using
	$\widehat V_n(0)=v_0^n$,
	\[
	\dot\gamma(0)=V(0)=\lim_{n\to\infty}v_0^n=v.
	\]
	
	It remains to pass to the limit in the calibration identity. Fix
	$a<b$ and set
	$k_a^n=\lceil a/\tau_n\rceil$,
	$k_b^n=\lfloor b/\tau_n\rfloor$.
	Summing the discrete calibration identities from $k_a^n$ to
	$k_b^n-1$ gives
	\begin{equation}
		\label{eq:summed-chain-calibration}
		u_{h_n,\tau_n}(x_{k_b^n}^n)
		-
		u_{h_n,\tau_n}(x_{k_a^n}^n)
		=
		\sum_{k=k_a^n}^{k_b^n-1}
		\tau_nL(x_k^n,v_k^n)
		-
		(k_b^n-k_a^n)\tau_n\bar L(h_n,\tau_n).
	\end{equation}
	
	Since $k_a^n\tau_n\to a$ and $k_b^n\tau_n\to b$, the locally uniform
	convergence of $\gamma_n$ and the uniform convergence of the discrete
	critical solutions imply
	\[
	u_{h_n,\tau_n}(x_{k_b^n}^n)
	-
	u_{h_n,\tau_n}(x_{k_a^n}^n)
	\to
	u(\gamma(b))-u(\gamma(a)).
	\]
	Moreover,
	\[
	\sum_{k=k_a^n}^{k_b^n-1}\tau_nL(x_k^n,v_k^n)
	=
	\int_{k_a^n\tau_n}^{k_b^n\tau_n}
	L\bigl(\gamma_n(\lfloor t\rfloor_n),V_n(t)\bigr)\,dt,
	\]
	where
	$\lfloor t\rfloor_n:=\tau_n\lfloor t/\tau_n\rfloor$.
	By the locally uniform convergence of $\gamma_n$,
	\eqref{eq:piecewise-velocity-convergence}, and the uniform velocity
	bound, the integrands converge uniformly on compact intervals; hence
	\[
	\sum_{k=k_a^n}^{k_b^n-1}\tau_nL(x_k^n,v_k^n)
	\to
	\int_a^bL(\gamma(t),\dot\gamma(t))\,dt.
	\]
	Finally,
	$(k_b^n-k_a^n)\tau_n\to b-a$ and
	$\bar L(h_n,\tau_n)\to-\alpha(H)$ by
	Theorem~\ref{thm:continuous-fully-discrete-critical-error}. Passing to
	the limit in \eqref{eq:summed-chain-calibration} yields
	\[
	u(\gamma(b))-u(\gamma(a))
	=
	\int_a^bL(\gamma(t),\dot\gamma(t))\,dt
	+
	\alpha(H)(b-a).
	\]
	Thus $\gamma$ is globally $u$-calibrated. By Tonelli regularity,
	$\gamma$ is an Euler--Lagrange trajectory, and therefore
	\[
	(x,v)
	=
	(\gamma(0),\dot\gamma(0))
	\in
	\widetilde{\mathcal I}(u)
	\subset
	\widetilde{\mathcal N}_L.
	\]
\end{proof}

\section*{Acknowledgments} The authors acknowledge the use of generative AI tools to assist with the graphical preparation of the figures and with the revision and execution of the code used in the numerical experiments. The mathematical formulation, the choice of the numerical parameters, and the verification and interpretation of all results were carried out by the authors.


\begin{thebibliography}{99}
		\small
		\setlength{\itemsep}{0.15em}
		\setlength{\parskip}{0pt}

		\bibitem{Arnaud2011}
		M.-C.~Arnaud,
		\newblock A non-differentiable essential irrational invariant curve for a
		$C^1$ symplectic twist map,
		\newblock \emph{Journal of Modern Dynamics}
		\textbf{5} (2011), no.~3, 583--591.

		\bibitem{Bernard2002}
		P.~Bernard,
		\newblock Connecting orbits of time dependent Lagrangian systems,
		\newblock \emph{Annales de l'Institut Fourier (Grenoble)}
		\textbf{52} (2002), no.~5, 1533--1568.

		\bibitem{BernardContreras2008}
		P.~Bernard and G.~Contreras,
		\newblock A generic property of families of Lagrangian systems,
		\newblock \emph{Annals of Mathematics}
		\textbf{167} (2008), no.~3, 1099--1108.

		\bibitem{BouillardFaouZavidovique}
		A.~Bouillard, E.~Faou, and M.~Zavidovique,
		\newblock Fast weak--KAM integrators for separable Hamiltonian systems,
		\newblock \emph{Mathematics of Computation}
		\textbf{85} (2016), no.~297, 85--117.
		
		\bibitem{Camilli}
		F.~Camilli, I.~Capuzzo-Dolcetta, and D.~L.~A.~Gomes,
		\newblock Error estimates for the approximation of the effective Hamiltonian,
		\newblock \emph{Applied Mathematics and Optimization}
		\textbf{57} (2008), no.~1, 30--57.

		\bibitem{CamilliMendico}
		F.~Camilli and C.~Mendico,
		\newblock Semi--discrete approximation of Aubry and Mather sets,
		\newblock arXiv:2604.24148, 2026.

		\bibitem{ContrerasDelgadoIturriaga}
		G.~Contreras, J.~Delgado, and R.~Iturriaga,
		\newblock Lagrangian flows: the dynamics of globally minimizing orbits II,
		\newblock \emph{Boletim da Sociedade Brasileira de Matem\'atica}
		\textbf{28} (1997), no.~2, 155--196.

		\bibitem{Fathi2008}
		A.~Fathi,
		\newblock \emph{Weak KAM Theorem in Lagrangian Dynamics},
		\newblock Cambridge Studies in Advanced Mathematics, vol.~88,
		\newblock Cambridge University Press, Cambridge, 2008.

		\bibitem{FiguerasHaroLuque}
		J.-Ll.~Figueras, A.~Haro, and A.~Luque,
		\newblock Rigorous computer-assisted application of KAM theory:
		a modern approach,
		\newblock \emph{Foundations of Computational Mathematics}
		\textbf{17} (2017), no.~5, 1123--1193.

		\bibitem{Garibaldi}
		E.~Garibaldi and P.~Thieullen,
		\newblock Minimizing orbits in the discrete Aubry--Mather model,
		\newblock \emph{Nonlinearity}
		\textbf{24} (2011), no.~2, 563--611.

		\bibitem{Gomes}
		D.~L.~A.~Gomes,
		\newblock Viscosity solution methods and the discrete Aubry--Mather problem,
		\newblock \emph{Discrete and Continuous Dynamical Systems}
		\textbf{13} (2005), no.~1, 103--116.

		\bibitem{Gomes_Ob}
		D.~L.~A.~Gomes and A.~M.~Oberman,
		\newblock Computing the effective Hamiltonian using a variational approach,
		\newblock \emph{SIAM Journal on Control and Optimization}
		\textbf{43} (2004), no.~3, 792--812.

		\bibitem{Hadikhanloo}
		S.~Hadikhanloo and F.~J.~Silva,
		\newblock Finite mean field games: fictitious play and convergence to a
		first order continuous mean field game,
		\newblock \emph{Journal de Mathématiques Pures et Appliquées}
		\textbf{132} (2019), 369--397.

		\bibitem{HaroDeLaLlave}
		A.~Haro and R.~de la Llave,
		\newblock A parameterization method for the computation of invariant tori
		and their whiskers in quasi-periodic maps: explorations and mechanisms for
		the breakdown of hyperbolicity,
		\newblock \emph{SIAM Journal on Applied Dynamical Systems}
		\textbf{6} (2007), no.~1, 142--207.

		\bibitem{HaroEtAl}
		A.~Haro, M.~Canadell, J.-Ll.~Figueras, A.~Luque, and J.-M.~Mondelo,
		\newblock \emph{The Parameterization Method for Invariant Manifolds:
			From Rigorous Results to Effective Computations},
		\newblock Applied Mathematical Sciences, vol.~195,
		\newblock Springer, Cham, 2016.

		\bibitem{HuguetDeLaLlaveSire}
		G.~Huguet, R.~de la Llave, and Y.~Sire,
		\newblock Computation of whiskered invariant tori and their associated
		manifolds: new fast algorithms,
		\newblock \emph{Discrete and Continuous Dynamical Systems}
		\textbf{32} (2012), no.~4, 1309--1353.

		\bibitem{Iturriaga-Mendico}
R.~Iturriaga, C.~Mendico, K.~Wang, and Y.~Xu,
\newblock Discretization and vanishing discount problems for first-order
mean field games,
\newblock \emph{Calculus of Variations and Partial Differential Equations}
\textbf{65}, 228 (2026).

		\bibitem{Iturriaga-Wang}
		R.~Iturriaga and K.~Wang,
		\newblock A discrete weak KAM method for first-order stationary mean field
		games,
		\newblock \emph{SIAM Journal on Applied Dynamical Systems}
		\textbf{22} (2023), no.~2, 1253--1274.

		\bibitem{karp}
		R.~M.~Karp,
		\newblock A characterization of the minimum cycle mean in a digraph,
		\newblock \emph{Discrete Mathematics}
		\textbf{23} (1978), no.~3, 309--311.

		\bibitem{MacKayGreene}
		R.~S.~MacKay,
		\newblock Greene's residue criterion,
		\newblock \emph{Nonlinearity}
		\textbf{5} (1992), no.~1, 161--187.

		\bibitem{MacKayMeissStark}
		R.~S.~MacKay, J.~D.~Meiss, and J.~Stark,
		\newblock Converse KAM theory for symplectic twist maps,
		\newblock \emph{Nonlinearity}
		\textbf{2} (1989), no.~4, 555--570.

		\bibitem{MacKayPercival}
		R.~S.~MacKay and I.~C.~Percival,
		\newblock Converse KAM: theory and practice,
		\newblock \emph{Communications in Mathematical Physics}
		\textbf{98} (1985), no.~4, 469--512.

		\bibitem{Mane1996}
		R.~Ma\~n\'e,
		\newblock Generic properties and problems of minimizing measures of
		Lagrangian systems,
		\newblock \emph{Nonlinearity}
		\textbf{9} (1996), no.~2, 273--310.

		\bibitem{Mather1991}
		J.~N.~Mather,
		\newblock Action minimizing invariant measures for positive definite
		Lagrangian systems,
		\newblock \emph{Mathematische Zeitschrift}
		\textbf{207} (1991), no.~2, 169--207.

		\bibitem{Rorro}
		M.~Rorro,
		\newblock An approximation scheme for the effective Hamiltonian and
		applications,
		\newblock \emph{Applied Numerical Mathematics}
		\textbf{56} (2006), no.~9, 1238--1254.
		
		\bibitem{Soga}
		K.~Soga, 
		\newblock Weak KAM theory for action minimizing random walks,  
		\newblock \emph{Calc. Var. Partial Differential Equations} {\bf 60} (2021), no.~5, Paper No. 179, 49 pp.

		\bibitem{Sorrentino2015}
		A.~Sorrentino,
		\newblock \emph{Action-Minimizing Methods in Hamiltonian Dynamics:
			An Introduction to Aubry--Mather Theory},
		\newblock Mathematical Notes, vol.~50,
		\newblock Princeton University Press, Princeton, 2015.

		\bibitem{Su}
		X.~Su and P.~Thieullen,
		\newblock Convergence of discrete Aubry--Mather models in the continuous
		limit,
		\newblock \emph{Nonlinearity}
		\textbf{31} (2018), no.~5, 2126--2155.
		
      \bibitem{tran}
        H.V.~Tran and  Y.F~Yu, 
        \newblock $L^\infty$ Variational Approximation of the Aubry Set,
        \newblock 	arXiv:2609.01557, 2026.
      

		\bibitem{Zavidovique}
		M.~Zavidovique,
		\newblock \emph{Discrete Weak KAM Theory},
		\newblock Lecture Notes in Mathematics, vol.~2377,
		\newblock Springer, Cham, 2025.

		\bibitem{Zhang2017}
		J.~Zhang,
		\newblock Generically Ma\~n\'e set supports uniquely ergodic measure for
		residual cohomology class,
		\newblock \emph{Proceedings of the American Mathematical Society}
		\textbf{145} (2017), no.~9, 3973--3980.

	\end{thebibliography}
\end{document}